\documentclass[fms,times]{cuparticle}

\usepackage{mathtools, amsmath, amssymb}
\usepackage{hyperref}
\usepackage[all]{xy}
\usepackage{color,graphicx}
\usepackage{enumitem}

\makeatletter
\renewcommand*\env@matrix[1][*\c@MaxMatrixCols c]{%
  \hskip -\arraycolsep
  \let\@ifnextchar\new@ifnextchar
  \array{#1}}
\makeatother

\newcommand{\on}[1]{\operatorname{#1}}
\newcommand{\bb}[1]{{\mathbb{#1}}}

\newcommand{\ca}[1]{{\mathcal{#1}}}

\newcommand{\mf}[1]{{\mathfrak{#1}}}
\newcommand*\isom{\xrightarrow{\sim}}
\newcommand{\rr}{\bb{R}}
\newcommand{\pair}[1]{\langle#1\rangle}
\newcommand{\divisor}{\operatorname{div}}
\newcommand{\ord}{\operatorname{ord}}
\newcommand{\Pic}{\operatorname{Pic}}

\newcommand{\rmH}{H}

\newcommand{\Aut}{\operatorname{Aut}}

\newcommand{\Ker}{\operatorname{Ker}}

\newcommand{\Sp}{\operatorname{Sp}}

\newcommand{\GL}{\operatorname{GL}}
\newcommand{\spann}{\operatorname{span}}

\newcommand{\abs}[1]{\lvert#1\rvert}
\newcommand{\aabs}[1]{\lvert\lvert#1\rvert\rvert}

\newcommand{\sing}{^{\textrm{sing}}}

\newcommand{\ra}{\rightarrow}

\newcommand{\sub}{\subseteq}

\newcommand{\Lie}{\operatorname{Lie}}

\newcommand{\ov}[1]{\overline{#1}}
\newcommand{\wt}[1]{\widetilde{#1}}
\newcommand{\Rat}{\operatorname{Rat}}

\renewcommand\Im{\operatorname{Im}}
\DeclareMathOperator{\Id}{Id}

\def\qq{\mathbb{Q}}

\def\rr{\mathbb{R}}
\def\zz{\mathbb{Z}}
\def\cc{\mathbb{C}}

\def\mm{\mathcal{M}}

\def\hh{\mathcal{H}}

\def\QQ{\boldsymbol{Q}}

\def\hh{\boldsymbol{H}}
\def\VV{\mathcal{V}}

\def\HH{\mathcal{H}}
\def\WW{\mathcal{W}}
\def\ww{{\boldsymbol{W}}}
\def\oo{\mathcal{O}}
\def\pp{\mathcal{P}}

\def\ff{\mathcal{F}}

\def\deldelbar{\partial \bar{\partial}}

\def\Ybar{\ov Y}
\def\Xbar{\ov X}

\def\Gr{\mathrm{Gr}}

\def\defeq{\stackrel{\text{\tiny def}}{=}}

\newdefinition{definition}{Definition}[section]
\newdefinition{conjecture}[definition]{Conjecture}
\newdefinition{condition}[definition]{Condition}
\newdefinition{problem}[definition]{Problem}
\newdefinition{assumption}[definition]{Assumption}

\newtheorem{proposition}[definition]{Proposition}
\newtheorem{lemma}[definition]{Lemma}
\newtheorem{theorem}[definition]{Theorem}
\newtheorem{corollary}[definition]{Corollary}
\newtheorem{claim}[definition]{Claim}
\newtheorem{inclaim}{Claim}[definition]

\newtheorem{remark}[definition]{Remark}
\newtheorem{example}[definition]{Example}

\newproof{proof}{Proof}

\numberwithin{equation}{section}

\newcounter{nootje}
\newcommand{\Lear}[1]{\left[#1 \vphantom{{#1}^{\sum}}\right]} 
\newcommand{\mat}[1]{\left(\begin{matrix}#1\end{matrix}\right)}

\volume{}
\doi{}

\begin{document}

\authorheadline{J.~I.~Burgos Gil, D.~Holmes and R.~de~Jong}
\runningtitle{Singularities of the biextension metric for abelian
  varieties}

\begin{frontmatter}

\title{Singularities of the biextension metric for families of abelian varieties}

\author[1]{Jos\'e~Ignacio~Burgos~Gil}
\address[1]{Instituto de Ciencias Matem\'aticas (CSIC-UAM-UCM-UCM3),
  Calle Nicol\'as Ca\-bre\-ra~15, Campus UAM, Cantoblanco, 28049 Madrid,
  Spain
\ead{burgos@icmat.es}}

\author[2]{David~Holmes}
\address[2]{Mathematical Institute,  
Leiden University,
PO Box 9512, 
2300 RA Leiden, 
The Netherlands
\ead{holmesdst@math.leidenuniv.nl}}

\author[3]{Robin~de~Jong}
\address[3]{Mathematical Institute,  
Leiden University,
PO Box 9512, 
2300 RA Leiden, 
The Netherlands
\ead{rdejong@math.leidenuniv.nl}}

\received{}

\begin{abstract} In this paper we study the singularities of the
  invariant metric of the Poincar\'e bundle over a family of abelian
  varieties and their duals over a base of arbitrary dimension. As an
  application of this study we prove the effectiveness of the height
  jump divisors for families of pointed abelian varieties. The effectiveness
  of the height jump divisor was conjectured by Hain in the more
  general case of variations of polarized Hodge structures of weight
  $-1$.     
\end{abstract}


\end{frontmatter}


\section{Introduction}

\subsection{Families of curves}
\label{sec:families_curves}
By way of motivation of the general results in this paper, consider
the following situation.  Let $X$ be a smooth complex algebraic
variety of dimension $n$, and let $\pi \colon Y\to X$ be a family of
smooth projective curves parametrized by $X$. Let $A$, $B$ be two
relative degree zero divisors on $Y \to X$, with disjoint support.  To
these divisors we can associate a function $h\colon X\to \rr$, given
by the archimedean component of the N\'eron height pairing
\begin{displaymath}
  h(x)=\langle A_x,B_x \rangle_{\infty} \, ,
\end{displaymath}
where $x \in X$.
Let $X\hookrightarrow \ov X$ be a smooth compactification of $X$ with
$D=\ov X\setminus X$ a normal crossings divisor. We are interested in
the behavior of the function $h$ close to the boundary divisor
$D$. As is customary to do, we assume that the monodromy operators on
the homology of the fibers of $Y \to X$ about all irreducible
components of $D$ are unipotent. Let $x_{0}$ be a point of $\ov X$,
and $U\isom \Delta^n$ a small enough coordinate neighborhood of
$x_{0}$ such that $ D\cap U$ is given by $q_1\cdots q_k=0$. Thanks to
a result of Brosnan and Pearlstein \cite{brospearl} (see also
\cite{hdj} \cite{lear} for the case where $X$ has dimension~$1$), there exist a
continuous function $h_{0}\colon U\setminus D\sing\ra \rr$ and
rational numbers $f_{1},\dots, f_k$ such that on $U\setminus D$ the equality
\begin{equation} \label{eq:2}
  h(q_1,\dots,q_n)=h_{0}(q_1,\dots,q_n)-\sum_{i=1}^{k}f_{i}\log|q_{i}|
\end{equation}
holds. Since $h_{0}$ is continuous on $U\setminus D\sing$, this
determines the behavior of $h$ close to the smooth points of $D$. The
question remains what happens when we approach a point of $D\sing$. In
other words, what kind of singularities may $h_{0}$ have on $D\sing$?

From work by G. Pearlstein \cite{pearldiff} we find a more precise
statement. Let $x_0 \in \ov X$ be as above. Then there exists a 
homogeneous weight one function $f \in \qq(x_1,\ldots,x_k)$ such that
the following holds. Consider a holomorphic test curve
$\ov \phi \colon \ov C\to \ov X$ that has image not contained in $D$,
a point $0\in \ov C$ such that $\ov \phi(0)=x_0$, and a local analytic
coordinate $t$ for $\ov C$ close to $0$. Assume that $\ov \phi$ is
given locally by
\begin{displaymath}
  t\mapsto\big(t^{m_{1}}u_{1}(t),\dots,t^{m_{k}}u_{k}(t),q_{k+1}(t),
  \dots,q_{n}(t)\big), 
\end{displaymath}
where $m_{1},\dots,m_{k}$ are non-negative integers, $u_{1},\dots,u_{k}$
are invertible functions and $q_{k+1},\dots,q_{n}$ are arbitrary holomorphic
functions. Then the asymptotic estimate 
\begin{equation} \label{onevariableasympt}
 h(\ov \phi(t)) =  b'(t) - f(m_1,\ldots,m_k) \log|t| 
\end{equation}
holds in a neighborhood of $0 \in \ov C$. Here $b'$ is a continuous
function that extends continuously over $0$.  

Naively one might expect
that the function $f$ is linear and $f(m_1,\ldots,m_k)$ is just a
linear combination of the numbers 
$f_{i}$ with coefficients given by the multiplicities $m_{i}$ of the
curve $\ov C$. In general, however this turns out not to be the
case. Examples of non-linear $f$ can be found in \cite{bhdj} and
\cite{bkk}. In \cite{abbf}, \cite{bhdj} and \cite{hdj} one finds a
combinatorial interpretation of the function $f$ in terms of potential
theory on the dual graphs of stable curves.

As a special case of one of the main results of this paper we will
have a stronger asymptotic estimate. Namely
\[ h(q_1,\ldots,q_n) = b(q_1,\ldots,q_n) + f(-\log|q_1|,\ldots,-\log|q_k|)    \]
on $U \setminus D$, where $b \colon U \setminus D \to \rr$ is a
bounded continuous function that extends in a continuous manner over  
$U \setminus D^{\mathrm{sing}}$. The boundedness of $b$ can be seen as
a uniformity property on the asymptotic estimates for different test
curves. In general, as shown by Example \ref{exm:1} below, the function  $b$
can not be extended continuously to $D^{\mathrm{sing}}$, thus the
boundedness of $b$ is the strongest estimate that can be hoped for.

As a concrete example of the shape of the function $f$, consider 
the stable curve $Y_0$ obtained by glueing two projective lines
at zero and infinity, and marking the point $(1:1)$ in both components. Let
$\Ybar 
\to \Xbar$ be a versal deformation of $Y_0$. The locus in $\Xbar$
where the morphism $\Ybar \to \Xbar$ is not smooth is a normal crossings divisor, locally
defined by $q_1q_2 = 0$, say. The examples in \cite{bhdj} and \cite{bkk}
show that the function $f(x_1, x_2)$ is given, up to linear forms in
$x_1$ and $x_2$,  by $x_1x_2/(x_1 + x_2)$.

One may ask for further properties of $h$. For example, a result of T. Hayama
and G.~Pearlstein \cite[Theorem~1.18]{hp} implies that
$h$ is locally integrable. Another question is whether the same can be
said about the forms $\partial h$ and $\deldelbar h$ and their
powers. As seen in \cite{bkk} in a case where $\Xbar$ is two-dimensional this may lead
to interesting intersection numbers between infinite towers of
divisors.  We plan to address this question in full generality in a subsequent
work. In this paper we will focus on the one-dimensional case because
it is the only case needed to treat Conjecture \ref{con:2} below. Thus assume that the dimension of $X$ is one.  Let $h_{0}$ be 
the function appearing in equation \eqref{eq:2}. Then we prove that the 1-form
$\partial h_{0}$ is locally integrable on $U$ with zero
residue. Also the 2-form $\partial\bar \partial h_{0}$ is locally
integrable on $U$.

\subsection{Admissible variations of mixed Hodge structures}
\label{sec:vari-hodge-struct}

The correct general setting for approaching these issues is to
consider a variation of polarized pure Hodge structures $\hh$ of
weight $-1$ over $X$, see for instance \cite{hainbiext} and
\cite{hain_normal}. Let $\hh^\lor$ be the dual variation. Let
$J(\hh) \to X$ and $J(\hh^\lor) \to X$ be the corresponding families
of intermediate jacobians. Then on
$J(\hh) \underset{X}{\times } J(\hh^\vee)$ one has a Poincar\'e
(biextension) bundle $\ca{P}=\ca{P}(\hh)$ with its canonical
(biextension) metric. The polarization induces an isogeny of complex
tori $\lambda \colon J(\hh) \to J(\hh^\lor)$. Let
$\nu, \mu \colon X \to J(\hh)$ be two sections (with good behavior
near $D$,  more precisely \emph{admissible normal functions}). Then we define
\[ L = \ca{P}_{\nu,\mu} \defeq (\nu,\lambda \mu)^* \ca{P} \, , \]
as a metrized analytic line bundle on $X$. We put
$\ca{P}_\nu=\ca{P}_{\nu,\nu}$. This ``diagonal'' case will be of
special interest to us. One important example, discussed at length in
\cite{hain_normal} and \cite{brospearl}, is given by the normal
function in $J(\bigwedge^3
H_1(Y_x))=H_{3}(J(Y_{x}))$ associated to the Ceresa cycle
$[Y_{x}]-[-Y_{x}]$ in $J(Y_x)$, for a family of curves $Y \to X$.

A second example is provided by the sections determined by two
relative degree zero divisors $A, B$ on a family of smooth projective curves,
as above.  
Let $\hh$ be the variation of Hodge structure given by the homology of the fibers of
the family of curves $Y \to X$. Then $J(\hh)$ is the usual jacobian
fibration associated to $Y \to X$. It is principally polarized in a
canonical way.  The divisors $A, B$ give rise to sections $\nu, \mu$
of $J(\hh) \to X$. The Deligne pairing associates to the line bundles
$\ca{O}_Y(A)$ and $\ca{O}_Y(B)$ a line bundle $\pair{A, B}$ on $X$, in a
functorial and bi-multiplicative way, see \cite{de}. The line bundle $\pair{A,B}$ comes with a
canonical rational section $s_{A,B}$, as well as a canonical hermitian
metric $\aabs{\cdot}_{A,B}$. The metric on $\pair{A,B}$ is determined
by the archimedean height pairing. More precisely, we have the identity 
\begin{equation*}
h(x) = \pair{A_x,B_x}_\infty = -\log \big(\aabs{s_{A,B}}_{A,B}(x)\big)
\end{equation*}
for all $x \in X$. There is  a canonical isometry
\[ \pair{A,B}^{\otimes (-1)} \isom \ca{P}_{\nu,\mu} \, . \]
Thus the singularity near $x_0$ of the biextension metric of the local
rational section $s_{A,B}$ precisely
gives the singularity of the function $h$ near $x_0$ as discussed
above.

Returning to the general set-up, the result of Brosnan and Pearlstein
\cite[Theorems 24 and 79]{brospearl} is that some power $L^{\otimes N}$ extends as a
continuously metrized line bundle over $\ov{X} \setminus
D^{\mathrm{sing}}$. Here we need to impose the condition that the
monodromy operators on the fibers of $\hh$ about all irreducible
components of $D$ are unipotent.
Moreover, Theorem 233 and Remark 234 of \cite{brospearl} provide a
canonical extension of $L^{\otimes N}$ on $\Xbar
\setminus D^{\on{sing}}$ to an analytic line bundle over the
whole of $\Xbar$ (though the metric will in general not extend
continuously over $D^{\on{sing}}$). Note that if the line bundle
$L^{\otimes N}$ on $\Xbar \setminus D^{\on{sing}}$ is algebraic, then it has a unique
extension to an algebraic line bundle on $\Xbar$. 
We denote the resulting line
bundle on $\Xbar$ by $\Lear{L^{\otimes N},\aabs{-}}_{\ov
X}$. This extension is commonly known as the \emph{Lear extension} of $L^{\otimes N}$, though the first general proof of its existence is due to Brosnan and Pearlstein in \cite{brospearl}. In order to remove the dependence on the choice of $N$ we will adopt the formalism of $\qq$-line bundles, and consider 
the Lear extension $\Lear{L,\aabs{-}}_{\ov
X}=\frac{1}{N} \Lear{L^{\otimes N},\aabs{-}}_{\ov
X}$ as a $\qq$-line bundle on~$\Xbar$. 

We are interested in the behavior of the biextension
metric on $L$ when we approach a point $x_0$ in the
singular locus $D^{\mathrm{sing}}$. Let $s$ be a 
section of $L=\ca{P}_{\nu,\mu}$ on $U \cap X$ that corresponds to an
admissible biextension variation of mixed Hodge structures. Pearlstein 
\cite[Theorem~5.19]{pearldiff} has proved that there exists a homogeneous weight one
function $f_s \in \qq(x_1,\ldots,x_k)$ such that for each holomorphic
test curve $\ov \phi \colon \ov C \to \ov X$  as above the asymptotic estimate
\begin{equation} \label{asymp_pearl} -\log\|s(\ov \phi(t))\| = b'(t) -f_s(m_1,\ldots,m_k)\log|t| 
\end{equation}
holds in a neighborhood $V$ of $0 \in \ov C$, with $b'(t)$ continuous on $V$. 

Now assume that the polarized variation $\hh$ is  torsion-free and of type $(-1,0)$,
$(0,-1)$ over $X$, so that the family $J(\hh) \to X$ is a family of polarized
abelian varieties over $X$. Under this assumption we are able to strengthen the result of
Pearlstein's.

\subsection{Statement of the main results}\label{sec:statement_of_main}

Recall that we work with a smooth complex algebraic variety $X$, provided with a
partial compactification $\ov X$ with $D= \ov X\setminus X$ a normal crossings divisor, and
a polarized pure variation of Hodge structures $\hh$ of weight $-1$ over $X$. 

Let $(q_1,\ldots,q_n) \colon U \isom \Delta^n$ be a
coordinate chart on $\ov{X}$ such that $D \cap U = \{ q_1\cdots q_k =
0 \}$. Denote by $D_{i}$ the local component of $D$ with equation given by
$q_{i}=0$.  For any $ 0 < \epsilon < 1$ write  
\[ 
U_\epsilon = \{ (q_1,\ldots,q_n) \in U : 
\abs{q_i} < \epsilon \quad \textrm{for all} \quad i=1,\ldots,n \} \,
. 
\]
Note that $U_\epsilon \cap X$ is identified via the coordinate
chart with $(\Delta^*_\epsilon)^k \times \Delta_\epsilon^{n-k}$.

\begin{theorem}  \label{singbiext} Assume that $\hh$ is a
  variation of torsion-free polarized pure Hodge structures of type
  $(-1,0), (0,-1)$ on $X$. Assume that the monodromy operators on the
  fibers of $\hh$ about the irreducible components of $D$ are
  unipotent. Let
  $\nu, \mu \colon X \to J(\hh)$ be two admissible normal functions of $J(\hh)$ over $X$. Then there
  exist an integer $d$, a homogeneous polynomial $Q\in
  \bb{Z}[x_1,\ldots,x_k]$ of degree $d$ with no zeroes on
  $\bb{R}_{>0}^k$ and, for each section $s$ of
  $\ca{P}_{\nu,\mu}$ corresponding to an admissible biextension
  variation of mixed Hodge structures over $U \cap X$, a homogeneous polynomial $P_s\in
  \bb{Z}[x_1,\ldots,x_k]$ of degree $d+1$ such that the homogeneous
  weight one rational function 
  $f_s=P_s/Q$ satisfies the following properties.
  \begin{enumerate}
  \item For all $\epsilon \in \bb{R}_{>0}$ small enough, the
    function
    \[ b(q_1,\ldots,q_n)=-\log\aabs{s} - f_s(-\log|q_1|,\ldots,-\log|q_k|) \]
    is bounded on $U_\epsilon \cap X$ and extends continuously over
    $U_\epsilon \setminus D^{\mathrm{sing}}$.
  \item The function $f_{s}$ is uniquely determined by the previous
    property. Moreover, if $s'$ is another section of $\ca{P}_{\nu,\mu}$ over $U \cap X$, such that
    \begin{equation}\label{eq:12}
      \divisor(s'/s)=\sum_{i=1}^{k}a_{i}D_{i} \, ,
    \end{equation}
    then the difference
    \begin{displaymath}
      f_{s'}-f_{s}=\sum_{i=1}^{k}a_{i}(-\log|q_{i}|)
    \end{displaymath}
    is linear in the functions $-\log|q_{i}|$.   
  \item The function $f_s \colon \bb{R}_{>0}^k \to \bb{R}$ extends to
    a continuous function $\ov{f}_s \colon \bb{R}^k_{\ge 0} \to
    \bb{R}$.
  \item \label{item:3} In the case that $\mu=\nu$, the function $f_s$ is convex as a
    function on $\bb{R}_{>0}^k$ and the function
    $\ov{f}_s$ is convex as a function on $\bb{R}^k_{\ge 0}$. 
  \end{enumerate}
\end{theorem}
We make a few remarks about Theorem \ref{singbiext}. First of all, by \cite[Theorem 81]{brospearl}, if $U$ is small enough,
    admissible sections $s$ as in Theorem \ref{singbiext} exist.
    
    Next, by \cite[Corollary 177]{brospearl}, if $U$ is small enough the set of admissible biextension variations on $U \cap X$ is a non-empty torsor over the group of meromorphic functions with
    poles  only on $D$. Hence the admissibility of $s$ and condition
    \eqref{eq:12} imply the admissibility of $s'$.
    
    Clearly, the function $f_s$ from Theorem \ref{singbiext} coincides with the $f_s$ from Pearlstein's asymptotic estimate (\ref{asymp_pearl}). However, we do not assume \cite[Theorem~5.19]{pearldiff} in our proof, hence our arguments give an independent proof of (\ref{asymp_pearl}) for the case of polarized, torsion-free variations of type $(-1,0), (0,-1)$.
    
    If the family $J = J(\bf H)$ of jacobians is algebraic, that is, $J$ is an
    abelian scheme over $X$, then any two algebraic sections $\mu$
    and $\nu$ of $J$ over $X$ are admissible, and for such $\mu, \nu$ the Lear extension
    of $\ca{P}_{\nu,\mu}$ over $\Xbar$ is an algebraic $\qq$-line bundle.  

Let the rank of $\hh$ be $2g$. Our proof of Theorem \ref{singbiext} in section \ref{sec:proof-main-results} will show that the function $f_s$ in the theorem has the shape 
\begin{equation} \label{shape_f_s}
{f}_s(x_1, \dots, x_k) = (\sum_{i=1}^k x_i A_i c_i )^t (\sum_{i=1}^k x_i A_i )^{-1} (\sum_{i=1}^k x_i A_i c_i ),
\end{equation}
where the $A_i$ $(i = 1, \dots, k)$ are positive semidefinite $g
\times g$ matrices such that $\sum_{i=1}^k A_{i}$ is positive definite, the
$c_i$ are in $\bb Q^g$, and are determined by 
the monodromy of $\mu$ and $\nu$  about the branches of the divisor $D$. Thus the singularity of $-\log\aabs{s}$ has the shape 
\begin{equation*}
\begin{split}
f_s(-\log &\abs{q_1}, \dots, -\log \abs{q_k}) \\
= &\left(\sum_{i=1}^k -\log\abs{q_i} A_i c_i \right)^t \left(\sum_{i=1}^k
  -\log\abs{q_i} A_i \right)^{-1} \left(\sum_{i=1}^k -\log\abs{q_i} A_i c_i
\right).  
\end{split}
\end{equation*}
Finally, Example \ref{exm:1} below will show that, in general, the locus of
indeterminacy $D^{\mathrm{sing}}$ of $b$ can not be reduced to a
smaller set.   

We next turn to the issue of local integrability, in dimension one. R. Hain has made the following conjecture 
(see \cite[Conjecture 6.4]{hain_normal}). Assume we work with an
arbitrary polarized variation of Hodge structures $(\hh,\lambda)$ of
weight $-1$, whose underlying local system of abelian groups is torsion-free, and let $\ca{P}$
be its  Poincar\'e bundle. Let $\nu$ be an admissible normal function of the family of intermediate jacobians $J(\hh)$ over $X$.
 \begin{conjecture}[Hain] \label{con:2}  Write $\hat{\ca{P}} =
   (\mathrm{id},\lambda)^*\ca{P}$ and let $\omega=c_1(\hat{\ca{P}})$ be
   the first Chern form of the pullback of the Poincar\'e bundle with its
   canonical metric. Assume that
   $X$ is a curve. Let $L=\ca{P}_\nu=\nu^* \hat{\ca{P}}$ with induced
   metric $\aabs{-}$ and let $N \in \zz_{>0}$ be such that $L^{\otimes
     N}$ extends as a continuous metrized line bundle over $\ov
   X$. Let $c_1\left(\Lear{L^{\otimes N}, \aabs{-}}_{\ov X}\right)$ be the
   first Chern class of the extended line bundle $\Lear{L^{\otimes N},
     \aabs{-}}_{\ov X}$. Then the $2$-form $\nu^*\omega$ is
   integrable on $\ov X$, and the 
   equality \[ \int_X \nu^*\omega = \frac{1}{N} \int_{\ov X} c_1
     \left(\Lear{L^{\otimes N}, \aabs{-}}_{\ov X} \right) \]
   holds. 
\end{conjecture}
Note that $\nu^*\omega = c_1(\ca{P}_\nu)$, and that the integral on
the right hand side equals $\frac{1}{N} \deg_{\ov X} \Lear{L^{\otimes
    N}, \aabs{-}}_{\ov X}$. We prove the following result, which
implies Hain's conjecture in the case of a variation of
torsion-free polarized 
Hodge structure of type $(-1,0), (0,-1)$. 
\begin{theorem} \label{localint} \label{theorem:local_integrability_over_curves}
  Assume that the polarized variation $\hh$ over $X$ is
  torsion-free and pure of type $(-1,0),
  (0,-1)$, and that the monodromy operators on the fibers of $\hh$
  about all irreducible components of $D$ are unipotent. Let $s$ be a
  section of $\ca{P}_{\nu,\mu}$ corresponding to an admissible
  biextension variation of mixed Hodge structures over $U \cap X$ and
  assume that $\dim X = 1$. Write
\[ -\log\|s\| = b(z) - r \log|t| \]
on $U \cap X$ with $r \in \bb{Q}$ and with $b$ bounded continuous on
$U$, as can be done by the existence of the Lear extension of
$\ca{P}_{\nu,\mu}$ over $\ov X$. Then the 1-form $\partial b$ is locally integrable
on $U$ with zero residue. Moreover the 2-form $\partial\bar \partial
b$ is locally integrable on $U$.
\end{theorem}
As also $\partial \bar{\partial} \log|t|$ is locally integrable, we find that $\partial \bar{\partial} \log \aabs{s}$ is locally integrable. 
Since moreover the $1$-form $\partial b$ has no residue on $U$, so
that $d [\bar{\partial} b]=[\partial \bar{\partial} b]$, upon
globalizing using bump functions and applying Stokes' theorem we find 
\[ \int_X c_1(\ca{P}_{\nu,\mu}) = \frac{1}{N} \int_{\ov X} c_1 \left(\Lear{\ca{P}_{\nu,\mu}^{\otimes N}, \aabs{-}}_{\ov X} \right) 
= \deg \Lear{\ca{P}_{\nu,\mu}, \aabs{-}}_{\ov X}\, . \]
In the diagonal case, we mention that by
\cite[Theorem~13.1]{hain_normal} or \cite[Theorem~8.2]{pearlpeters}
the metric on $\ca{P}_\nu$ is non-negative. Thus Theorem \ref{localint} implies that actually the inequality 
\begin{equation} \label{positivedegree}
\on{deg} \Lear{\ca{P}_\nu, \aabs{-}}_{\ov{X}} \geq 0  
\end{equation}
holds. We mention that in a letter to P. Griffiths, G. Pearlstein sketches a proof of Conjecture \ref{con:2}, and hence of the inequality (\ref{positivedegree}), without the assumption that the type be $(-1,0), (0,-1)$.

We return again to the setting where the parameter space $X$ is of any dimension. However, we specialize to the ``diagonal'' case where $\mu=\nu$. Consider, as before a test curve $\ov \phi \colon \ov C\to \ov X$ that has image not contained in $D$, and a point $0\in \ov C$ such that $\ov \phi(0)=x_0$. Let $\phi$ denote the restriction of $\ov \phi$ to $\ov C \setminus \ov \phi^{-1}D$. The $\qq$-line bundle
\begin{equation*}
\Lear{\phi^*(\ca{P}_\nu, \aabs{-})}_{\ov{C}}^{\otimes -1} \otimes \ov{\phi}^*\Lear{\ca{P}_\nu, \aabs{-}}_{\ov{X}}
\end{equation*}
has a canonical non-zero rational section, as it is canonically
trivial over $\ov C \setminus \ov \phi^{-1}D$. We call its divisor the
height jump divisor $J=J_{\phi,\nu}$ on $\ov C$. R. Hain has made the following conjecture (see \cite[end of \S 14]{hain_normal}).
\begin{conjecture} \label{con:1} For all holomorphic test curves 
$\ov \phi \colon \ov C\to \ov X$ with image not contained in $D$, the height jump divisor $J=J_{\phi,\nu}$ on $\ov C$ is effective.
\end{conjecture}
 Choose coordinates in a neighbourhood $U$ of $x_{0}$ as before so that $x_{0}$ has coordinates $(0,\dots,0)$
and let $f_s \in \bb{Q}(x_1,\ldots,x_k)$ be as in Pearlstein's asymptotic estimate (\ref{asymp_pearl}), based on the choice of some admissible section $s$ of $\ca{P}_\nu$ on $U \cap X$. It can be shown that the function $f _s\colon \bb{R}_{>0}^k \to \bb{R}$ extends to a continuous function $\ov{f}_s \colon \bb{R}^k_{\ge 0} \to \bb{R}$. Locally around $0$ the map $\ov
\phi $ can be written as
\begin{displaymath}
  \ov \phi (t)=(t^{m_{1}}u_{1}(t),\dots,
  t^{m_{k}}u_{k}(t),q_{k+1}(t),\dots, q_{n}(t)), 
\end{displaymath}
where, for $i\in [1,k]$, $m_{i}> 0$ and $u_{i}(0)\not = 0$.
 Write
$\ov{f}_{s,i}\defeq\ov{f}_s(0,\ldots,0,1,0,\ldots,0)$ (the $1$ placed
in the $i$-th spot),  then
\begin{equation} \label{expression_J}
  \ord_0 J =-\ov{f}_s(m_1,\ldots,m_k)+\sum_{i=1}^{k}m_{i}\ov{f}_{s,i} \, . 
\end{equation}
Note that indeed $\ord_0 J$ is independent of the choice of $s$. The rational number $\ord_0 J$ is called the ``height jump''
associated to the test curve $\ov \phi$, the admissible normal function $\nu$ and the point $0\in \ov C$. 

The terminology is due to R.\ Hain \cite{hain_normal}, who also observed a first instance where the height jump is non-zero. We refer the reader to the monograph \cite{brospearl} by P. Brosnan and
G. Pearlstein, where an extensive study of
the height jump in complete generality is given. Note that
the height jumps precisely when $f_s$ is not linear.
We mention that Conjecture \ref{con:1} about the height jump was stated in
\cite{hain_normal} only for the normal function on $\mm_g$ associated to the Ceresa
cycle, but it seems reasonable to make this broader conjecture.  

In this paper we prove Conjecture \ref{con:1} in the case of admissible normal functions of families of polarized abelian varieties. 
\begin{theorem} \label{theorem:effectivity} Assume that the polarized
  variation $\hh$ over the smooth complex variety $X$ is torsion-free and pure of type
  $(-1,0), (0,-1)$, and that the monodromy operators on the fibers of
  $\hh$ about all irreducible components of $D$ are unipotent. Let $\nu$ be an admissible normal function of the family of intermediate jacobians $J(\hh)$ over $X$. Then
  for all holomorphic test curves $\ov \phi \colon \ov C\to \ov X$
  with image not contained in $D$, the associated height jump divisor
  $J=J_{\phi,\nu}$ on $\ov C$ is effective.
\end{theorem}
Combining with inequality (\ref{positivedegree}) we obtain
\begin{corollary} Assume that $\ov C$ is smooth and projective. Then under the assumptions of Theorem \ref{theorem:effectivity}, the $\bb{Q}$-line bundle $\ov{\phi}^*\Lear{\ca{P}_\nu, \aabs{-}}_{\ov{X}}$ has non-negative degree on $\ov{C}$.
\end{corollary}
The key to our proof of Theorem \ref{theorem:effectivity} is the convexity of  the homogeneous function $\ov f_s$, as asserted in Theorem \ref{singbiext}(\ref{item:3}). We have the following explicit expression for $\ord_0 J$. In equation (\ref{shape_f_s}) we already gave an expression for $f_s$ and hence $\ov f_s$ in terms of matrices $A_i$ and vectors $c_i$ for $i=1,\ldots,k$. We will see in subsection \ref{propertiesnormlike} that $\ov{f}_{s,i}=c_i^t A_i c_i $ for $i=1,\ldots,k$. Following the general expression (\ref{expression_J})  this gives
\[ \ord_0 J = -(\sum_{i=1}^k m_i A_i c_i )^t (\sum_{i=1}^k m_i A_i )^{-1} (\sum_{i=1}^k m_i A_i c_i ) + \sum_{i=1}^k m_i c_i^t A_i c_i   \]
for the height jump in our setting. 

Turning again to the case of the Ceresa cycle, note that since the
intermediate Jacobian of the primitive part of $H_{3}(J(Y_{x}))$ is a
compact complex torus but not an abelian variety, we can not apply
directly our
results for families of abelian varieties to this case. 

In the special case of families of jacobians of curves Conjecture \ref{con:1} has been proved in \cite{bhdj}. The proof in this special case makes heavy use of the combinatorics of dual graphs of nodal curves, and so cannot readily be extended to families of abelian varieties, nor does it seem practical to reduce the general case to that of jacobians of curves.

\begin{remark} After the initial submission of the present paper to arxiv, two proofs of Conjecture \ref{con:1} have appeared, see \cite{brospearl} and \cite{BGHdJ2}. 
\end{remark}

\subsection{Overview of the paper}

We review the content of the different sections of this paper. In the preliminary section \ref{sec:three-faces-coin} we start by recalling the notions of $\qq$-line bundle and of Lear extension, and the Poincar\'e bundle on the product of a complex
torus and its dual, together with its associated metric. We
also recall the explicit description of the Poincar\'e bundle and its
metric on a family of polarized abelian
varieties. Also we study the period map associated
to a family of pointed polarized abelian varieties. Moreover we give a local expansion for the metric on the pullback of the Poincar\'e bundle under this period map. The functions that appear as the
logarithm of the norm of a section of the pullback of the Poincar\'e
bundle will be called norm-like functions.

In section \ref{technical} we study norm-like functions and give
several estimates on their growth and that of their derivatives. 
Finally in section \ref{sec:proof-main-results} we prove the main results on local integrability and positivity of the height jump. 

We fix some notation that we will use throughout. Let $r$ be a positive integer. For any commutative ring $R$ we
will denote by $\on{Col}_{r}(R)$ (respectively $\on{Row}_{r}(R)$,
$M_{r}(R)$ and $S_{r}(R)$) the set of column vectors of size $r$ with
entries in $R$ (respectively row vectors, matrices and symmetric matrices of size $r$-by-$r$). 

We denote by $S_{r}^{++}(\bb{R})\subset S_{r}(\bb{R})$ (respectively $S_{r}^{+}(\bb{R})\subset S_{r}(\bb{R})$) the cone of
positive definite (respectively positive semidefinite) symmetric
real matrices. We denote by $\mathbb{H}_r$ Siegel's upper half
space of rank $r$, and by $\mathbb{P}^r$ its compact dual. 
 
By a variety we mean an integral separated scheme of finite type over $\bb{C}$. 

\section{Preliminary results}
\label{sec:three-faces-coin}

\subsection{Lear extensions}	
\label{sec:lear-extensions}

We start by recalling the formalism of $\bb{Q}$-line bundles. Details can be found in \cite[Definition 2.10]{bhdj}.

\begin{definition}
Let $X$
be a complex variety. An (algebraic resp.\ analytic)
$\bb{Q}$-line bundle over $X$ is a pair $(L,r)$ where $L$ is an
(algebraic resp.\ analytic) line
bundle on $X$ and $r>0$ is a positive integer (informally, we think of
it as $L^{\otimes{1/r}}$). A metrized 
$\bb{Q}$-line bundle is a triple $(L,\aabs{-},r)$, where $(L,r)$ is a
$\bb{Q}$-line bundle and $\aabs{-}$ is a continuous metric on $L$. An
isomorphism of $\bb{Q}$-line bundles $(L_{1},r_{1})\to (L_{2},r_{2})$ is
an equivalence class of pairs $(a,f)$ where $a$ is a positive integer
and $f\colon L_{1}^{\otimes ar_2}\to L_{2}^{\otimes ar_1}$ is an
isomorphism, where the 
equivalence relation is generated by setting $(a,f) \sim (an, 
f^{\otimes n})$.    An  isomorphism of metrized line bundles is an
isometry if one (equivalently all) of the
corresponding morphisms of line bundles is an isometry. Every line
bundle $L$ gives rise to a $\bb{Q}$-line bundle $(L,1)$. Note that, if
$L$ is a line bundle and $r>1$ is an integer, then there is a
canonical isomorphism $(L^{\otimes r},r)\simeq (L,1)$. Moreover, if
$L$ is a torsion line bundle so that $L^{\otimes r}\simeq \ca{O}_{X}$, then
there is an isomorphism of $\bb{Q}$-line bundles $(L,1)\to
(\ca{O}_{X},r)$. If we do not need to specify the multiplicity $r$, a
$\bb{Q}$-line bundle will be denoted by a single letter. We note that
the group of isomorphism classes of $\bb{Q}$-line bundles on $X$ is equal to $\Pic(X)
\otimes_{\bb{Z}} \bb{Q}$. 

We denote
\begin{displaymath}
 \Rat_{\qq}(X)=\left(\oo(X)\setminus\{0\},\times\right)\otimes \qq. 
\end{displaymath}
If $(L,r)$ is a $\qq$-line bundle, a $\qq$-rational section of $(L,r)$ (or
rational section for short) is an
equivalence class of symbols
\begin{math}
  s^{\frac{1}{rd}},
\end{math}
where $s$ is a non-zero rational section of $L^{\otimes d}$. Two
symbols $s_{1}^{1/rd_{1}}$ and $s_{2}^{1/rd_{2}}$ 
are equivalent if
\begin{displaymath}
  (s_{1})^{\otimes d_{2}}
  =(s_{2})^{\otimes d_{1}}
\end{displaymath}
as a section of $L^{\otimes d_{1}d_{2}}$.
The space of rational sections of $(L,r)$ is a torsor over
$\Rat_{\qq}(X)$. Moreover, if $s$ and $s'$ are rational sections of
$(L,r)$ and $(L',r')$ then $s\otimes s'$ is a rational section of
$(L^{\otimes r'}\otimes (L')^{\otimes r},rr')$, but there is no additive structure of
rational sections.

The \emph{divisor of the section} $s^{\frac{1}{rd}}$ is
\begin{displaymath}
  \divisor(s^{\frac{1}{rd}})=\frac{1}{rd}\divisor(s).
\end{displaymath}
\end{definition}

\begin{definition}[Lear extension]
Let $X \sub \ov{X}$ be an open immersion of smooth complex varieties,
such that the boundary divisor $D \defeq \ov{X} \setminus X$ has normal
crossings, and $L$ a line bundle on $X$ with continuous metric
$\aabs{-}$. A \emph{Lear extension} of $L$ is a $\bb{Q}$-line bundle
$(\ca{L},r)$ on $\ov{X}$ together with an isomorphism $\alpha \colon
(L,1)\to (\ca{L},r)|_{X}$ and a continuous metric on
$\ca{L}|_{\ov{X} \setminus D^\mathrm{sing}}$ such that the isomorphism
$\alpha $ is an isometry. Since $D^\mathrm{sing}$ has codimension at
least $2$ in $\ov{X}$, if a Lear extension exists then it is unique up to a
unique isomorphism. If a Lear extension of $L$ exists we denote it by 
$\Lear{L, \aabs{-}}_{\ov{X}}$. Note that the isomorphism class of the
Lear extension of $L$ depends not only on $L$ but also on the metric on $L$. 

If $s$ is a rational section of $L$, writing $s=(s^{\otimes r})^{\frac{1}{r}}$, it can
also be seen as a 
rational section of $\Lear{L, \aabs{-}}_{\ov{X}}$. We will denote by
$\on{div}_{X}(s)$ the divisor of $s$ as a rational section of $L$
and by $\on{div}_{\ov X}(s)$ the divisor of $s$ as a rational section of
$\Lear{L,  \aabs{-}}_{\ov{X}}$.  
\end{definition}

\subsection{Poincar\'e bundle and its metric}
\label{sec:poincare-bundle}

In this section we recall the definition of the Poincar\'e bundle and
its biextension metric. Moreover we make the biextension metric explicit in the case of families of polarized abelian varieties.
 
In the literature one can find small discrepancies in the description
of the Poincar\'e bundle, see Remark \ref{rem:1}. These discrepancies
can be traced back to two different choices of the identification of a
complex torus with its bidual. Moreover, there are also different
conventions regarding the sign of the polarization of the abelian
variety. Since one of our main results is a positivity result it is worthwhile
to fix all the signs to avoid these ambiguities.

\medskip
\noindent
\emph{Complex tori and their duals.}
Let $g\ge 0$ be a non-negative integer, $V$ a $g$-dimensional complex vector
space and $\Lambda \subset V$ a
rank $2g$ lattice. The quotient $T=V/\Lambda $ is a compact complex
torus. It is a K\"ahler complex manifold, but in general it is
not an algebraic variety.

We recall the construction of the dual torus of $T$. We denote by
$V^{\ast}=\on{Hom}_{\ov{\bb{C}}}(V,\bb{C})$ the space of 
antilinear forms $w\colon V\to \bb{C}$. This is not the dual
$V^{\vee}$ of $V$. In fact, let $\ov V$ denote the abelian group $V$
with the complex structure $\ov \cdot$ given by
\begin{displaymath}
  \alpha \,\ov \cdot \,v = \ov \alpha \cdot v.
\end{displaymath}
Then $V^{\ast}=\ov V^{\vee}$.

The bilinear form
\begin{displaymath}
  \langle\cdot,\cdot\rangle\colon V^{\ast}\times V\to \bb{R},\ \langle
  w,z \rangle \defeq \on{Im}(w(z)) 
\end{displaymath}
is non-degenerate. Thus
\begin{displaymath}
  \Lambda ^{\vee}\defeq \{\lambda \in V^{\ast}\mid \langle
  \lambda ,\Lambda  \rangle\subset \bb{Z}\}
\end{displaymath}
is a lattice of $V^{\ast}$. The lattice $\Lambda ^{\vee}$ is
canonically isomorphic to the dual of the lattice $\Lambda $. The
quotient $T^{\vee}=V^{\ast}/\Lambda ^{\vee}$ is 
again a compact complex torus, called the \emph{dual torus} of $T$. 

We
can identify $V$ with $\on{Hom}_{\ov{\bb{C}}}(V^{\ast},\bb{C})$ by the
rule 
\begin{equation}\label{eq:3}
  z(w)=\ov{w(z)}
\end{equation}
so that the bilinear pairing 
\begin{displaymath} 
  (V^{\ast}\oplus V)\otimes (V^{\ast}\oplus V)\to \rr,\quad
  (w,z)\otimes(w',z')\mapsto \on{Im}(w(z'))+\on{Im}(z(w'))  
\end{displaymath}
is antisymmetric. With this identification the double dual
$(T^{\vee})^{\vee}$ gets
identified with $T$. 

The points of $T^{\vee}$ define homologically trivial line bundles on
$T$ giving an isomorphism of $T^{\vee}$ with
$\on{Pic^{0}}(T)$. We recall this construction. Let $\bb{C}_{1}$
denote the subgroup of $\bb{C}^{\times}$ of elements of norm one.
Let $w\in
V^{\ast}$. Denote by $[w]$ its class in $T^{\vee}$ and by
$\chi_{[w]}\in \on{Hom}(\Lambda ,\bb{C}_{1})$ 
the character 
\begin{equation}\label{eq:4}
  \chi_{[w]}(\mu)=\exp(2\pi i\langle w,\mu \rangle). 
\end{equation}
The line bundle associated to $[w]$ is the line bundle $L_{[w]}$ with automorphy
factor $\chi_{[w]}$. In other words, consider the action of $\Lambda $
on $V\times \bb{C}$ given by
\begin{displaymath}
  \mu (z,t)=(z+\mu ,t\exp(2\pi i \langle w,\mu \rangle)).
\end{displaymath}
Write $L_{[w]}=(V\times \bb{C})/\Lambda $. The projection $V\times
\bb{C}\to V$ induces a map $L_{[w]}\to T$. It is easy to check that
$L_{[w]}$ is a holomorphic line bundle on $T$ that only depends on
the class $[w]$. Note that the identification between $T^{\vee}$ and
$\on{Pic^{0}}(T)$ is not completely canonical because it depends on a
choice of sign. We could equally well have used the character
$\chi_{[w]}^{-1}$.  

\medskip
\noindent
\emph{The Poincar\'e bundle.} 
Note that, although the cocycle \eqref{eq:4} is not holomorphic
in $w$, the line bundle
$L_{[w]}$ varies holomorphically with $w$, defining a holomorphic line
bundle on $ T\times T^{\vee}$ called the Poincar\'e bundle. See
\cite[\S~2.5]{bl} for details.  

\begin{definition}\label{def:2}
  A \emph{Poincar\'e (line) bundle} $\ca{P}$ is a holomorphic
  line bundle on $T\times T^{\vee}$ that satisfies
  \begin{enumerate}
  \item the restriction $\ca{P}|_{T\times \{[w]\}}$ is isomorphic to $L_{[w]}$;
  \item the restriction $\ca{P}|_{ \{0\}\times T^{\vee}}$ is trivial. 
  \end{enumerate}
  A \emph{rigidified Poincar\'e bundle} is a Poincar\'e bundle
  together with an isomorphism $\ca{P}|_{\{0\}\times T^{\vee}}
  \isom \ca{O}_{ \{0\}\times T^{\vee}}$.
\end{definition}

To prove the existence of a Poincar\'e bundle, consider the
map
\begin{displaymath}
 a_{\ca{P}}\colon (\Lambda \times \Lambda^{\vee})\times (V\times
V^{\ast}) \to \bb{C}^{\times} 
\end{displaymath}
given by
\begin{equation}\label{eq:16}
  a_{\ca{P}}((\mu,\lambda ),(z,w))
  =\exp\Big(\pi \big((w+\lambda )(\mu )+\overline{\lambda(z)}\big)\Big).
\end{equation}
This map is holomorphic in $z$ and $w$. Moreover, since for $(\mu
,\lambda )\in \Lambda \times \Lambda ^{\vee}$, 
\begin{displaymath}
  \langle \lambda ,\mu \rangle=\frac{1}{2i}(\lambda (\mu
  )-\overline{\lambda (\mu )})\in \bb{Z},
\end{displaymath}
the map $a_{\ca{P}}$ is a cocycle for the additive action of $\Lambda
\times \Lambda ^{\vee}$ on $V\times V ^{\ast}$. Hence, it is an automorphy factor that
defines a holomorphic line bundle $\ca{P}$ on $T\times T ^{\vee}=V\times
V ^{\ast}/\Lambda \times \Lambda  ^{\vee}$. 

For a fixed $w\in V^{\ast}$,
\begin{displaymath}
  a_{\ca{P}}((\mu,0),(z,w))
  =\exp(\pi w(\mu)).
\end{displaymath}
This last cocycle is equivalent to the cocycle \eqref{eq:4}. Indeed,
\begin{displaymath}
  \exp(\pi w(\mu ))\exp(\pi \overline{w(z+\mu)})^{-1}\exp(\pi  \overline{w(z)})= 
  \exp(2\pi i \langle w,\mu \rangle ),
\end{displaymath}
and the function $z\mapsto \exp(\pi  \overline{w(z)})$ is holomorphic
in $z$. 
Thus the restriction $\ca{P}|_{T\times \{[w]\}}$ is isomorphic to
$L_{[w]}$. Moreover
\begin{displaymath}
  a_{\ca{P}}((0,\lambda),(0,w))=1,
\end{displaymath}
which implies that the restriction $\ca{P}|_{\{0\}\times T^{\vee}}$ is
trivial. The uniqueness of the Poincar\'e bundle follows from the
seesaw principle (see \cite[Appendix~A]{bl}).

We conclude
\begin{proposition}\label{prop:6}
  A Poincar\'e bundle exists and is unique up to isomorphism. A
  rigidified Poincar\'e bundle exists and is unique up to a unique
  isomorphism.   
\end{proposition}

\begin{remark}\label{rem:1}
Using the above identification of $T$ with the dual torus
of $T^{\vee}$ we have that, for a fixed $z\in V$, the restriction
$\ca{P}|_{\{[z]\}\times  T^{\vee}}$ agrees with $L_{[z]}$. In fact    
\begin{displaymath}
  a_{\ca{P}}((0,\lambda ),(z,w))
  =\exp(\pi \overline{\lambda (z)}),
\end{displaymath}
and, arguing as in the proof of Proposition \ref{prop:6}, this cocycle 
is equivalent to the cocycle
\begin{displaymath}
  \exp(2\pi i \on{Im}(\ov{\lambda(z) }))=\exp(2\pi i \langle z,\lambda
  \rangle ).
\end{displaymath}
 Note that the definition of the Poincar\'e bundle in \cite[\S
 3.2]{hainbiext} states that $\ca{P}|_{\{[z]\}\times
   T^{\vee}}=L_{[-z]}$. The discrepancy between \cite{hainbiext} and
 the current paper is due to a different choice of
 identification between $T$ and $(T^{\vee})^{\vee}$.
\end{remark}

\begin{remark}\label{rem:2} As we will see later, in equation \eqref{eq:13}, the
  cocycle \eqref{eq:16} is not optimal because it does not vary
  holomorphically in holomorphic families of tori.
\end{remark}

\medskip
\noindent
\emph{Group theoretical interpretation of the
  Poincar\'e bundle.}
We next give a group theoretic description of
the Poincar\'e 
bundle. We start with the additive real Lie group $W$ given by
\begin{displaymath}
  W= V\times V^{\ast}.
\end{displaymath}
Denote by $\wt {W}$ the semidirect product $\wt{W}=W\ltimes
\bb{C}^{\times}$,  where the product in $\wt{W}$ is given by 
\begin{equation} \label{eq:7}
  \big((z,w),t\big)\cdot\big((z',w'),t'\big)=
  \big((z+z',w+ w'),tt'\exp(2\pi i \langle w,z'\rangle )\big). 
\end{equation}
Clearly the group
\begin{equation}\label{eq:15}
  W_{\bb{Z}}=\Lambda \times \Lambda ^{\vee}
\end{equation}
is a subgroup of $\wt{W}$.

Consider the space
\begin{equation}\label{eq:14}
  P\defeq V\times V^{\ast}\times 
  \bb{C}^{\times} 
\end{equation}
and the action of $\wt{W}$ on $P$ by biholomorphisms given by
\begin{equation} \label{eq:6}
  \big((\mu,\lambda ),t\big)\cdot ((z,w),s)
  =\big(z+\mu,w+\lambda,
  ts \exp(\pi (w+\lambda )(\mu )+\pi \overline{\lambda (z)})\big). 
\end{equation}
The projection $P\to V\times V^{\ast}$ induces a map
$W_{\bb{Z}}\backslash P\to T\times T^{\vee}$. The action of
$\bb{C}^{\times }$ on $P$ by acting on the third factor provides
$W_{\bb{Z}}\backslash P$ with a structure of
$\bb{C}^{\times}$-bundle over $T\times T^{\vee}$. Denote by 
$\ca{P}_{T}=(W_{\bb{Z}}\backslash
P)\underset{\bb{C}^{\times}}{\times } \bb C$
the associated holomorphic line bundle. The structure of
$P$ as a product space induces a canonical rigidification
$\ca{P}_{T}|_{ \{0\}\times T^{\vee}}=\ca{O}_{\{0\}\times T^{\vee}}$.

\begin{proposition}\label{prop:2} The line bundle
  $\ca{P}_{T}$ is a rigidified Poincar\'e line bundle. 
\end{proposition}
\begin{proof}
From the explicit
description of the cocycle \eqref{eq:16} and of the action
\eqref{eq:6}
we deduce that $\ca{P}_{T}$ is a Poincar\'e bundle.
\end{proof}

\medskip
\noindent
\emph{The metric of the Poincar\'e bundle.}
The Poincar\'e bundle has a metric that is determined up to constant
by the condition that its curvature form is invariant under
translation. On a rigidified Poincar\'e bundle, with given rigidification
$\ca{P}_{T}|_{ \{0\}\times T^{\vee}} \isom  \ca{O}_{\{0\}\times
  T^{\vee}}$, the constant is 
fixed by imposing the condition  $\| 1\|=1$. We now describe
explicitly this metric.

Let  $\wt{W}_{1}=W\ltimes
\bb{C}_{1}$ with the product described before.
Denote by $\ca{P}^{\times}_{T}$ the Poincar\'e bundle
with the zero section deleted. Since
$\ca{P}_{T}^{\times}= W_{\bb{Z}}\backslash P$, the invariant metric of
$\ca{P}_{T}$ is described by the unique function
$\|\cdot\|\colon P\to \bb{R}_{>0}$ satisfying the conditions
\begin{enumerate}
\item (Norm condition) For $(z,w,s )\in P$, we have
  \begin{displaymath}
    \|(z,w,s )\|=|s |\|(z,w,1)\|.
  \end{displaymath}
\item (Invariance under $\wt{W}_{1}$) For $g\in \wt{W}_{1}$ and $x\in
  P$, we have 
  \begin{displaymath}
    \|g\cdot x\|=\|x\|
  \end{displaymath}
\item (Normalization) $\|(0,0,1)\|=1$.
\end{enumerate}
Using the explicit description of the action given in \eqref{eq:6}, we
have that
\begin{displaymath}
  (z,w,s)=(z,w,1)\cdot (0,0,s\exp(-\pi w(z))),
\end{displaymath}
from which
one easily derives that the previous conditions imply 
\begin{equation}\label{eq:8}
  \|(z,w,s)\|^{2}= |s|^{2}\exp\Big(-\pi \big(w(z)+\overline{w(z)}\big)\Big).
\end{equation}

\medskip
\noindent
\emph{Holomorphic families of complex tori.} Let $X$ be a
complex manifold and $\ca{T} \to X$ a holomorphic family 
of dimension $g$ complex tori.
This means that $\ca{T}$ is defined by
a holomorphic vector bundle $\VV$ of rank $g$ on $X$ and an integral local
system $\Lambda \subset \VV$ of rank $2g$ such that, for each $s\in X$, the fiber $\Lambda _{s}$ is a lattice in $\VV_{s}$ and the flat
sections of $\Lambda $ are holomorphic sections of $\VV$. Indeed
$\Lambda $ is the local system  $s \mapsto H_{1}(\ca{T}_{s},\bb{Z})$
and $\VV$ the holomorphic vector bundle $s\mapsto
H_{1}(\ca{T}_{s},\bb{C})/ F^{0}H_{1}(\ca{T}_{s},\bb{C})$. 

 We now want to give to the dual family of compact tori a holomorphic
 structure. That is, we want to
construct a holomorphic family of compact tori $\ca{T}^{\vee}$ with a
canonical identification
$(\ca{T}^{\vee})_{s}=(\ca{T}_{s})^{\vee}$. This construction is not
completely obvious because the vector spaces $(\VV_{s})^{\ast}$ vary
anti-holomorphically with $s$. We will use the lattice $\Lambda $ to
define a holomorphic structure on this family of vector spaces.

Write $\HH_{\cc}=\Lambda \otimes \ca{O}_{X}$. It is a holomorphic
vector bundle, with a holomorphic surjection $\HH_{\cc}\to \VV$ and an
integral structure that determines a complex conjugation in
$\HH_{\cc}$. The kernel $\ca{F}^{0}= \Ker(\HH_{\cc}\to \VV)$ is a
holomorphic vector bundle. For every $s\in X$, the surjection
$\HH_{\cc}\to \VV$ allows us to identify $\ov {\ca{F}^{0}}_{s}$ with
$\VV_{s}$, hence $\ca{F}^{0}_{s}$ with $\ov \VV_{s}$. Let $\Lambda
^{\vee}$ be the dual local system to $\Lambda $. On the dual vector bundle
$\HH^{\vee}= \Lambda ^{\vee}\otimes  \ca{O}_{X}$ consider the
orthogonal complement $(\ca{F}^{0})^{\perp}$ to $\ca{F}^{0}$. Then
$(\ca{F}^{0})^{\perp}$ is isomorphic with the dual vector bundle
$\VV^{\vee}$. The quotient
\begin{math}
  \HH^{\vee}/(\ca{F}^{0})^{\perp}
\end{math}
is a holomorphic vector bundle that we denote by $\VV^{\ast}$. The
identification $\ca{F}^{0}_{s}=\ov \VV_{s}$ gives us the equality
\begin{displaymath}
  (\VV^{\ast})_{s}=
  (\HH^{\vee}/(\ca{F}^{0})^{\perp})_{s}=(\ca{F}^{0}_{s})^{\vee}=
  (\ov \VV_{s})^{\vee}=(\VV_{s})^{\ast},
\end{displaymath}
 that explains the notation. 

Then the dual family of tori is defined as 
\begin{displaymath}
  \ca{T}^{\vee}=\VV^{\ast}/\Lambda ^{\vee}.
\end{displaymath}

Let $U\subset X$ be a small enough open subset such
that the restriction of $\ca{T}$ to $U$ is topologically
trivial. Choose $s_0\in U$ and an integral basis  
$$(a,b)=(a_{1},\dots,a_{g},b_{1},\dots,b_{g})$$ of
$\Lambda _{s_{0}}$ such that $(a_{1},\dots,a_{g})$ is
a complex basis of $\VV_{s_0}$. By abuse of
notation, we denote by $a_{i},b_{i}$, $i=1,\dots,g$ the corresponding
flat sections of $\Lambda $. We can see them as holomorphic sections
of $\ca{H}_{\bb{C}}$ and we will also denote by $a_{i}, b_{i}$ their
images in $\VV$. After shrinking $U$ if
necessary, we can assume that the sections $a_{i}$ form a frame of
$\VV$, thus we can write
  \begin{equation}\label{eq:10}
    (b_{1},\dots,b_{g})=(a_{1},\dots, a_{g})\Omega 
  \end{equation}
 for a holomorphic map $\Omega \colon U \to M_{g}(\bb{C})$. We call
 $\Omega$ the period matrix of the variation on the basis
 $(a,b)$. Note that 
 condition \eqref{eq:10} is equivalent to saying that
 $\ca{F}^{0}\subset \ca{H}_{\bb{C}}$ is
 generated by the columns of the matrix
 \begin{displaymath}
   \begin{pmatrix}
     -\Omega \\
     \on{Id}
   \end{pmatrix}.
 \end{displaymath}
Writing $\HH_{\bb{R}}$ for the real vector subbundle of $\HH_{\bb{C}}$ formed by sections that are invariant under complex conjugation, 
we have that $\ca{F}^{0}\cap \HH_{\bb{R}}=0$. This implies that $\Im
\Omega$ is non-degenerate.  The complex basis
 $(a_{1},\dots,a_{g})$ gives us an  
 identification of $\VV|_U$ with the trivial vector bundle
 $\on{Col}_{g}(\bb{C})$ and the basis
 $(a,b)$ identifies $\Lambda $ with the trivial
 local system $\on{Col}_{g}(\bb{Z})\oplus \on{Col}_{g}(\bb{Z})$. With
 these identifications, the
inclusion $\Lambda \to \VV$ is given by 
\begin{displaymath}
  (\mu _{1},\mu _{2})\mapsto \mu=\mu _{1}+\Omega \mu _{2}.
\end{displaymath}

Let now $(a^{\ast},b^{\ast})=(a_{1}^{\ast},\dots,a_{g}^{\ast},b_{1}^{\ast},\dots
,b_{g}^{\ast})$ be the basis of $\Lambda ^{\vee}_{s_{0}}$ dual to
$(a,b)$.  As before we extend the elements
$a_{i}^{\ast},b_{i}^{\ast}$, $i=1,\dots,g$ to flat sections of
$\Lambda $ over $U$.  Then $b_{1}^{\ast},\dots ,b_{g}^{\ast}$ is a frame 
of $\VV^{\ast}$. One can check that, on
$\ca{V}^{\ast}$, the equality
\begin{displaymath}
  (a_{1}^{\ast},\dots,a_{g}^{\ast})=-(b_{1}^{\ast},\dots,b_{g}^{\ast})\Omega ^{t}
\end{displaymath}
holds. Thus if we identify $\VV^{\ast}$ with the trivial vector bundle
$\on{Row}_{g}(\bb{C})$
using the basis $(b^{\ast})$ and $\Lambda ^{\vee}$ with the trivial
local system
$\on{Row}_{g}(\bb{Z})\oplus \on{Row}_{g}(\bb{Z})$ using the basis
$(a^{\ast},b^{\ast})$ we obtain that the
inclusion $\Lambda^{\vee}\to \VV^{\ast}$ is given by
\begin{equation}\label{eq:9}
  (\lambda _{1},\lambda _{2})\mapsto \lambda =-\lambda _{1}\Omega +\lambda
  _{2}. 
\end{equation}
In the fixed bases, one can check that 
the pairing between $\ca{V}^{\ast}$ and $\ca{V}$
is given by 
\begin{equation}\label{eq:11}
  w(z)=-w (\Im \Omega)^{-1} \bar {z},
\end{equation}
where $w \in \on{Row}_{g}(\bb{C})$ and $z \in \on{Col}_{g}(\bb{C})$,
while 
the pairing between the
lattice $\Lambda $ and its dual $\Lambda ^{\vee}$ is given by
\begin{displaymath}
  \langle (\lambda _{1},\lambda _{2}),(\mu _{1},\mu _{2})\rangle =
  \lambda _{1}\mu _{1}+ \lambda _{2}\mu _{2},
\end{displaymath}
where $\lambda_1, \lambda_2 \in \on{Row}_{g}(\bb{Z})$ and $\mu_1,
\mu_2 \in \on{Col}_{g}(\bb{Z})$. Clearly the pairing between $\Lambda
$ and $\Lambda ^{\vee}$ has integer values.

The cocycle $a_{\ca{P}}$ from equation
\eqref{eq:16} can now be written down explicitly as 
\begin{multline}\label{eq:13}
  a_{\ca{P}}((\mu _{1},\mu _{2}),(\lambda _{1},\lambda _{2}),(z,w))\\=
  \exp(-\pi ((w-\lambda _{1}\Omega +\lambda _{2})(\Im \Omega)^{-1}(\mu
  _{1}+\bar \Omega 
  \mu _{2})+ (-\lambda _{1}\bar \Omega +\lambda _{2})(\Im \Omega)^{-1}z)),
\end{multline}
which is not holomorphic with respect to $\Omega $. Thus it
does not give us on the nose a holomorphic Poincar\'e bundle in
families. Nevertheless 
the construction of the Poincar\'e bundle can be given a holomorphic structure.

\begin{proposition}\label{prop:1}
  Let $X$ be a complex manifold and $\ca{T}\to X $ a holomorphic
  family of dimension $g$ complex tori. Let $\nu_0\colon X\to
  \ca{T}\underset{X}{\times }\ca{T}^{\vee}$ be the zero section. 
Then
  \begin{enumerate}
  \item the fiberwise dual tori form a
    holomorphic family of complex tori $\ca{T}^{\vee}\to X$;
  \item \label{item:2} on $\ca{T}\underset{X}{\times }\ca{T}^{\vee}$ there is a holomorphic
    line bundle $\ca{P}$, together with an isomorphism
    $\nu_0^{\ast}\ca{P}\isom \ca{O}_{X} $, called the
    rigidified Poincar\'e bundle, which is unique up to a unique
    isomorphism, and is
    characterized by the property that
for every point $p\in X$, the restriction 
      $\ca{P}|_{\ca{T}_{p}\times \ca{T}^{\vee}_{p}}$ is the
      rigidified Poincar\'e bundle of $\ca{T}_{p}$; 
  \item there is a unique metric on $\ca{P}$ that induces the trivial
    metric on $\nu_0^{\ast}\ca{P}=\ca{O}_{X} $ and whose curvature is
    fiberwise translation invariant.
  \end{enumerate}
\end{proposition}
\begin{proof} Fix an open subset $U\subset X$ as before. 
  The dual family of tori $\ca{T}^{\vee}$ is holomorphic by definition.

  In order to prove that the Poincar\'e bundle defines a holomorphic
  line bundle on the family
  we need to exhibit a new cocycle that is holomorphic in $z$, $w$ and
  $\Omega $ and that, for fixed $\Omega $, is equivalent to
  $a_{\ca{P}}$ holomorphically in $z$ and $w$. Write $\lambda =-\lambda
  _{1}\Omega +\lambda _{2}$ and $\mu =\mu _{1}+\Omega
  \mu _{2}$ as before with $\lambda_1, \lambda_2 \in \on{Row}_{g}(\bb{Z})$ and $\mu_1, \mu_2 \in \on{Col}_{g}(\bb{Z})$. Consider the cocycle
  \begin{equation}\label{eq:5}
    b_{\ca{P}}((\lambda,\mu),(z,w))
    =\exp(2 \pi i((w-\lambda _{1}\Omega +\lambda _{2})\mu _{2}-\lambda_{1} z))
  \end{equation}
for $w \in \on{Row}_{g}(\bb{C})$ and $z \in \on{Col}_{g}(\bb{C})$. Then $b_{\ca{P}}$ is holomorphic in $z$, $w$, and $\Omega $. Consider also the function
  \begin{equation}\label{eq:19}
    \psi (z,w)=\exp(-\pi  w(\Im \Omega)^{-1}z),
  \end{equation}
  which is holomorphic in $z$
  and $w$. Since
  \begin{displaymath}
    b_{\ca{P}}((\mu,\lambda),(z,w))=a_{\ca{P}}((\mu,\lambda
    ),(z,w))\psi (z,w)\psi (z+\mu ,w+\lambda )^{-1}
  \end{displaymath}
  we deduce that the cocycle $b_{\ca{P}}$ determines a line bundle
  that satisfies the properties stated in item (\ref{item:2}) from the proposition over the
  open $U$. The
  uniqueness follows again from the seesaw principle. 
    By the uniqueness, we can glue together the rigidified Poincar\'e bundles obtained
  in different open subsets $U$ to obtain a rigidified Poincar\'e bundle over $X$.
  
  The fact that the invariant metric has invariant curvature fixes it
  up to a function on $X$ that is determined by the normalization
  condition. Thus if it exists, it is unique. Since the expression for
  the metric in  \eqref{eq:8} is smooth in $\Omega $ and the change of
  cocycle function in  \eqref{eq:19} is also smooth in $\Omega $ we
  obtain an invariant metric locally. Again the uniqueness implies that we can
  patch together the different local expressions. 
\end{proof}

\begin{remark} \label{formulanorm}
  Since the cocycle $a_{\ca{P}}$ does not vary holomorphically in
  families, the frame for the Poincar\'e bundle used in equation
  \eqref{eq:8} is not holomorphic in families. The cocycle
  $b_{\ca{P}}$ and the rigidification do determine a holomorphic frame
  of the Poincar\'e bundle over $X\times V\times
  V^{\ast}$. In this holomorphic frame  the metric is given by
  \begin{align}
    \|(z,w,s)\|^{2}&=|s|^{2}\exp(-\pi (w(z)+\overline{w(z)}))|\psi
                     (z,w)|^{2}\notag \\
    &=|s|^{2}\exp(4\pi \on{Im}(w)(\Im \Omega)^{-1} \on{Im}(z))\label{eq:20},
  \end{align}
  where $\psi $ is the function given in \eqref{eq:19}. 
\end{remark}

\medskip
\noindent
\emph{Abelian varieties.}
We now specialize to the case of polarized abelian
varieties. A polarization on the torus $T=V/\Lambda $ is the datum of an antisymmetric
non-degenerate bilinear form $E\colon \Lambda \times \Lambda \to \bb{Z}$
such that for all $v, w \in V$,
\begin{displaymath}
  E(iv,iw)=E(v,w),\qquad -E(iv,v)>0,\text{ for }v\not=0.
\end{displaymath}
Here we have extended $E$ by $\bb{R}$-bilinearly to $V=\Lambda \otimes \bb{R}$. Note that the standard convention in the literature on abelian
varieties is to ask $E(iv,v)$ to be positive. But this convention is
not compatible with the usual convention in the literature on Hodge
Theory. We have changed the sign here to have compatible conventions
for abelian varieties and for Hodge structures. 
   
Since $E$ is antisymmetric and non-degenerate we can choose an integral
basis $(a,b)$ such
that the matrix of $E$ on $(a,b)$ is given by
\begin{equation} \label{type}
  \begin{pmatrix}
    0 & \Delta \\
    -\Delta  & 0
  \end{pmatrix} \, ,
\end{equation}
where $\Delta $ is an integral diagonal matrix. We will call such
basis a $\qq$-symplectic integral basis. From a $\qq$-symplectic
integral basis $(a,b)$ we can construct a symplectic rational basis
$(a\Delta^{-1} ,b)$.  

With the choice of a $\qq$-symplectic integral basis, the
condition $E(iv,iw)=E(v,w)$ is equivalent to
 the product matrix $\Delta \Omega $ being symmetric. Thus
 $\Omega^{t}\Delta =\Delta \Omega $. The condition
 $-E(iv,v)>0$ is equivalent 
to $\Delta \Im \Omega$ being 
positive definite. This last condition is equivalent to that any of
the symmetric matrices $(\Im \Omega)^{t}\Delta $, $((\Im
\Omega)^{-1})^{t}\Delta 
$ or $\Delta (\Im \Omega)^{-1}$ is positive definite.  

Recall from (\ref{eq:10}) that $\Omega \in
M_g(\bb{C})$ is determined by the relation $b=a\Omega$. The
polarization $E$ defines a positive definite hermitian form $H$ on 
$V$ given by
\begin{displaymath}
  H(v,w)=-E(iv,w)-iE(v,w),
\end{displaymath}
so that we recover the polarization $E$ as the restriction of
$-\on{Im}(H)$ to $\Lambda \times \Lambda $. In the basis
$(a_{1},\dots,a_{g})$ of $V$, 
the hermitian form $H$ is given by $\Delta (\Im \Omega)^{-1}=((\Im
\Omega)^{-1})^{t}\Delta $. That is, under 
the identification $V=\on{Col}_{g}(\bb{C})$, we have
\begin{equation}\label{eq:21}
  H(v,w)=v^{t}\Delta (\Im \Omega)^{-1}\ov{w}. 
\end{equation}

The polarization defines an isogeny $\lambda_{E}\colon T\to T^{\lor}$
that is given by the map $V\to V^{\ast}$, $v\mapsto H(v,-)$. Under
the identification
$V^{\ast}=\on{Row}_{g}(\bb{C})$ given by the basis $(b^{\ast})$, by
equations \eqref{eq:11} and \eqref{eq:21}, we deduce that 
$\lambda_E$ is given by
\begin{equation}\label{eq:22}
  \lambda _{E}(v)=-v^{t}\Delta .
\end{equation}
The fact that $\Delta \Omega $ is symmetric and $\Delta $ is integral 
implies that this map sends
$\Lambda $ to $\Lambda ^{\lor}$ defining an isogeny. The dual
polarization $E^{\lor}$ on
$V^{\ast}$ is given by the hermitian form $H^{\vee}(e,f)=e(\Im
\Omega)^{-1}\Delta ^{-1}\ov{f}^{t}$ 
so that the  map $V\to V^{\ast}$ is an isometry. 

Consider now the composition of the diagonal map with the polarization
map on the second factor $(\mathrm{id},\lambda_E) \colon T\to T\times
T^{\vee}$ and let $\ca{P}$ be the Poincar\'e bundle
on $T\times T^{\vee}$. Then $(\mathrm{id},\lambda_E)^{\ast}\ca{P}$
is an ample line bundle on $T$ whose first Chern class agrees with the
given polarization of $T$. 
\begin{theorem} \label{explicitmetricPoinc}
The metric induced on
the bundle $(\mathrm{id},\lambda_E)^{\ast}\ca{P}$ is given by the function $\|\cdot\|\colon V\times \bb{C}^{\times}\to \rr_{>0}$,
\begin{equation}\label{eq:23}
  \|(z,s)\|^{2}=|s|^{2}\exp(-4\pi \Im(z)^{t}\Delta (\Im \Omega)^{-1}\Im(z)).
\end{equation}
\end{theorem}
\begin{proof} This follows from equations \eqref{eq:20} and \eqref{eq:22}.
\end{proof}

\medskip
\noindent
\emph{Hodge structures of type $(-1,0),(0,-1)$.} Recall that a pure Hodge
structure of type $(-1,0), (0,-1)$ is given by
\begin{enumerate}
\item A finite rank $\bb{Z}$-module, $H_{\bb{Z}}$.
\item A decreasing filtration $F^{\bullet}$ on $H_{\bb{C}}\defeq
  H_{\bb{Z}}\otimes \bb{C}$ such that
  \begin{displaymath}
    F^{-1}H_{\bb{C}}=H_{\bb{C}},\quad F^{1}H_{\bb{C}}=0,\quad
    H_{\bb{C}}=F^{0}H_{\bb{C}}\oplus \overline{F^{0}H_{\bb{C}}}.
  \end{displaymath}
\end{enumerate}

A polarization of a Hodge structure of type $(-1,0),(0,-1)$
is a non-degenerate antisymmetric bilinear form $Q\colon
H_{\bb{Z}}\otimes H_{\bb{Z}}\to \bb{Z}$ which, when extended to
$H_{\bb{C}}$ by linearity, satisfies the ``Riemann bilinear relations''
\begin{enumerate}
\item The subspace $F^{0}H_{\bb{C}}$ is isotropic.
\item If $x\in F^{0}H_{\bb{C}}$, then $iQ(x,\ov x)>0$. 
\end{enumerate}

We will be interested only in torsion-free Hodge structures.
We recall that the category of torsion-free Hodge structures of type
$(-1,0),(0,-1)$ and the category of compact complex tori are
equivalent cf. \cite[Exercise 1.5.10]{bl}. If
$H=(H_{\bb{Z}},F^\bullet)$
is such a Hodge structure, we write $V=H_{\bb{C}}/F^{0}$ and $\pi
\colon H_{\bb{C}}\to H_{\bb{C}}/F^{0}$ for the projection. Then
$\Lambda \defeq \pi (H_{\bb{Z}})$ is a lattice in $V$, that defines
a torus $T=V/\Lambda $. This torus is denoted $J(H)$ and called the
Jacobian of $H$.

Conversely, if $T$ is a complex torus,
then $H_{1}(T,\bb{Z})$ is torsion-free and has a Hodge structure of type
$(-1,0),(0,-1)$. 

If $(H_{\bb{Z}},F^\bullet)$ has a polarization $Q$ then, identifying
$\Lambda $ with $H_{\bb{Z}}$ and writing $E=Q$, we obtain a 
polarization of $T$. We finish by verifying that, indeed $E$ is a
polarization in the sense of complex tori. That $E$ is non-degenerate
follows from the non-degeneracy of $Q$. Let $v,w\in V$, choose $\bar
x,\bar y\in
\overline{F^{0}H_{\bb{C}}}$
such that $\pi (\bar x)=v$ and $\pi (\bar y)=w$. Write $x$, $y$ for the
complex conjugates of $\bar x$  and $\bar y$ respectively. Then
$x+\bar x\in
H_{\bb{Z}}\otimes \bb{R}$ and $\pi (x+\bar x)=v$, while
$ix-i\bar x\in H_{\bb{Z}}\otimes \bb{R}$  and $\pi (ix-i\bar
x)=-iv$. Thus by the first Riemann bilinear relation
\begin{align*}
  E(iv,iw)&=Q(-ix+i\bar x,-iy+i\bar y)=Q(x,\bar y)+Q(\bar x,y)\\
  E(v,w)&=Q(x+\bar x,y+\bar y)=Q(x,\bar y)+Q(\bar x,y),
\end{align*}
Thus $E(iv,iw)=E(v,w)$.
Moreover, by the second bilinear relation
\begin{displaymath}
  H(v,v)=-E(iv,v)=-Q(-ix+i\bar x,x+\bar x)=2iQ(x,\bar x)>0.
\end{displaymath}

\subsection{Nilpotent orbit theorem}
\label{sec:norm-section}
The aim of this section is to formulate a version of the Nilpotent
orbit theorem that allows us to deal with variations of mixed Hodge structures,
in a setting with several variables. Such a Nilpotent orbit theorem is
stated and proved in \cite{pearlhiggs}. In order to formulate this
theorem, we need quite a bit of background material and in particular
define the notion of ``admissibility'' for variations of mixed Hodge
structures. Also we need to take a detailed look at the behaviour of
monodromy on the fibers of the underlying local systems. Most of the
introductory material below is taken from \cite[Section~14.4]{ps} and
\cite{pearlhiggs}.  

\medskip
\noindent
\emph{Variations of polarized mixed Hodge structures.}
Let $X$ be a complex manifold. A graded-polarized variation of mixed
Hodge structures on $X$ is a local system $\hh \to X$ of finitely
generated abelian groups equipped with: 
\begin{enumerate}
\item A finite increasing filtration
\[ \ww_\bullet \colon \quad 0 \subseteq \ldots \subseteq \ww_k
  \subseteq \ww_{k+1} \subseteq \ldots \subseteq \hh_\qq \]
of $\hh_\qq = \hh \otimes \qq$ by local subsystems, called the weight filtration,
\item A finite decreasing filtration
\[ \ff^\bullet \colon \quad \hh_\cc \otimes \oo_X \supseteq \ldots
  \supseteq \ff^{p-1} \supseteq \ff^p \supseteq \ldots \supseteq 0 \] 
of the vector bundle $\HH=\hh_\cc \otimes \oo_X$ by holomorphic
subbundles, called the Hodge filtration, 
\item For each $k \in \zz$ a non-degenerate bilinear form
\[ \boldsymbol{Q}_k \colon \Gr_k^\ww(\hh_\qq) \otimes \Gr_k^\ww(\hh_\qq) \to \qq_X  \]
of parity $(-1)^k$,
\end{enumerate}
such that: 
\begin{enumerate}
\item For each $p \in \zz$ the Gauss-Manin connection $\nabla$ on
  $\HH$ satisfies the ``Griffiths transversality condition'' $\nabla
  \ff^p \subseteq \Omega^1_X \otimes \ff^{p-1}$,
\item For each $k \in \zz$ the triple
  $(\Gr_k^\ww(\hh_\qq),\ff^\bullet\Gr_k^\ww(\HH),\boldsymbol{Q}_k)$ is
  a variation of pure polarized rational Hodge structures of weight
  $k$. Here for each $p \in \zz$ we write $\ff^p \Gr_k^\ww(\HH)$ for
  the image of $\ff^p \HH \cap \ww_k \HH$ in $\Gr_k^\ww(\hh_\cc)$
  under the projection map $\ww_k \HH \to \Gr_k^\ww(\hh_\cc)$. 
\end{enumerate}
A variation of polarized mixed Hodge structures will be called
torsion-free if $\hh$ is a local system of torsion-free abelian
groups. A $\qq$-variation of polarized mixed Hodge structures is
defined analogously with the difference that $\hh$ is a local system
of finite dimensional $\qq$-vector spaces.

\medskip
\noindent 
\emph{Period domains.}
If $(H,W_\bullet,F^\bullet)$ is a mixed Hodge
structure, then $H_\cc$ has a unique bigrading $I^{\bullet,\bullet}$
such that 
\[ F^pH_\cc = \oplus_{r \geq p,s} I^{r,s} \, , \quad
W_k H_\cc = \oplus_{r+s \leq k} I^{r,s} \, , \quad
I^{r,s} = \overline{I}^{s,r} \bmod \oplus_{p<r,q<s} I^{p,q} \, . \]
The integers $h^{r,s}=\dim I^{r,s}$ are called the Hodge numbers of
$(H,W_\bullet,F^\bullet)$. 

Given a quadruple $(H,W_\bullet,Q_k,h)$ with $H$ a rational vector space,
$W_\bullet$ an increasing filtration of $H$,  $Q_k$ a collection of
non-degenerate bilinear forms of parity $(-1)^k$ on $\Gr_k^W(H)$,  and
a partition of $\dim (H)$ into a sum of non-negative
integers $h=\{ h^{r,s} \}$ satisfying the symmetry condition
$h^{r,s}=h^{s,r}$, there exists a natural classifying space 
(also known as a period domain) $\mm=\mm(h)=\mm(H,W_\bullet,Q_k,h)$ of
mixed Hodge structures 
$(W_\bullet, F^\bullet)$ on $H$ which are graded-polarized by $Q_k$.

We recall the construction of $\mm$ from \cite[\S~3]{pearlhiggs}. Write
\begin{displaymath}
  f^{p}=\sum_{r\ge p,\, s}h^{r,s}\quad \text{and}\quad f_{k}^{p}=\sum_{r\ge p}h^{r,k-r}
\end{displaymath}
and let $\check{\mm}$ be the set of all decreasing filtrations
$F^{\bullet}$ of $H_{\cc}$ satisfying
\begin{displaymath}
  \dim(F^{p})=f^{p},\  \dim(F^{p}\Gr_{k}^{W})=f^{p}_{k},\text{ and }
  Q_{k}(F^{p}\Gr_{k}^{W},F^{k-p+1}\Gr_{k}^{W})=0
\end{displaymath}
The group
\begin{displaymath}
  G_{\cc}=\left\{ g\in \GL(H_{\cc})^{W} \, \middle | \, \Gr_{k}(g)\in
    \Aut_{\cc}(Q_{k})\right\} 
\end{displaymath}
is a complex algebraic group that acts transitively on $\check{\mm}$
giving to it a structure of complex manifold. The manifold
$\check{\mm}$ is usually called the ``compact dual'' of $\mm$ by
analogy with the pure case, although in general it is not compact.

The period domain
$\mm$ is the subset of $\check{\mm}$ formed by the filtrations
$F^{\bullet}$ such that $(H,W_{\bullet},F^{\bullet},Q)$ is a polarized
$\qq$-mixed Hodge structure. By \cite[Lemma 3.9]{pearlhiggs} $\mm$ is
an open subset of $\check{\mm}$, hence it has an induced structure
of complex manifold. By the same lemma, the group
\begin{displaymath}
  G_{P}=\left\{ g\in \GL(H_{\cc})^{W} \, \middle | \, \Gr_{k}(g)\in
    \Aut_{\rr}(Q_{k})\right\}
\end{displaymath}
acts transitively on $\mm$.
We also consider the group
\begin{equation}\label{eq:17}
  G_{\rr}=\left\{ g\in \GL(H_{\rr})^{W} \, \middle | \, \Gr_{k}(g)\in
    \Aut_{\rr}(Q_{k})\right\}. 
\end{equation}
Note that we have inclusions
\[ G_{\rr} \subset G_P \subset G_\cc \, . \]
\begin{remark} \label{rem:4}
  The group $G_{\rr}$ acts transitively on the subset $\mm_{\rr}$ of filtrations
  defining a mixed Hodge structure that is split over $\rr$.
  If the filtration $W$ has only two non-trivial weights that are
  adjacent, that is, if there is a $k$ such that
  \begin{displaymath}
    0= W_{k-2}\subset W_{k-1} \subset W_{k}=H,
  \end{displaymath}
  then any mixed Hodge structure on $\mm(H,W,Q,h)$ is split over
  $\rr$. Therefore $\mm_{\bb{R}}=\mm$ and $G_{\rr}$ acts transitively on $\mm$. This will hold for
  the case of interest to us in Section \ref{sec:families}. 
\end{remark}

\noindent
\emph{Relative filtrations.}
Let $H$ be a rational vector space, equipped with a finite increasing
filtration $W_\bullet$. We let $N$ denote a nilpotent endomorphism of
$H$, compatible with $W_\bullet$. We call an increasing filtration
$M_\bullet$ of $H$ a weight filtration for $N$ relative to $W_\bullet$
if the two following conditions are satisfied:
\begin{enumerate}
\item for each $i \in \zz$ we have $NM_i \subseteq M_{i-2}$,
\item for each $k \in \zz$ and each $i \in \mathbb{N}$ we have that
  $N^i$ induces an isomorphism
\[ N^i \colon \Gr_{k+i}^M \Gr_k^W H \isom \Gr_{k-i}^M \Gr_k^W H \]
of vector spaces.
\end{enumerate}
It can be verified that if $H$ has a weight filtration for $N$
relative to $W_\bullet$, then it is unique. We call $N$ strict if
$N(H)\cap W_k = N(W_k)$ for all $k \in \zz$. By \cite[Proposition
2.16]{sz}, if the filtration $W_\bullet$ has length two (in the sense
that $H=W_k$ and $W_{k-2}=0$ for some $k$), and if $H$ has a weight
filtration for $N$ relative to $W_\bullet$, then $N$ is strict.

\medskip
\noindent
\emph{Admissible variations of mixed Hodge structures.}
Now let $(\hh,\ww_\bullet,\ff^\bullet,\boldsymbol{Q}_k)$ be a
variation of graded-polarized mixed Hodge structures over the
punctured unit disc $\Delta^*$.  Let $s_{0}\in
\Delta ^{\ast}$  and $(H,W_\bullet,F^\bullet,Q_k)$ the fiber of
$(\hh,\ww_\bullet,\ff^\bullet,\boldsymbol{Q}_k)$ over $s_{0}$.
Let $\gamma $ be
a generator of the fundamental group $\pi _{1}(\Delta ^{\ast},s_{0})$
and $T$ the monodromy operator defined by $\gamma $ acting on $H$. 
Since $\ww_\bullet$ is a filtration by
local subsystems, the monodromy operator preserves $\ww_\bullet$.
The operator $T$ can be written as $T=T_{s}T_{u}$, where $T_{s}$ is
semisimple and  $T_{u}$ is unipotent. The monodromy is said to be
quasi-unipotent if $T_{s}^{r}=\Id$ for certain integer $r\ge 1$.
We denote by $N=\log
T_{u}$ the logarithm of the unipotent part of the monodromy. Clearly
$N$ is nilpotent. 

By \cite[II Remarque 5.5]{dediff} the vector bundle $\HH=\hh \otimes_{\bb{C}} \ca{O}_{\Delta^*}$ can be
``canonically'' extended to a vector bundle $\tilde{\HH}$ on the unit
disk. Moreover, the subbundles $\WW_{k}=\ww_{k}\otimes \oo_{\Delta ^{\ast}}$
also extend canonically to subbundles $\tilde{\WW}_{k}$ of $\tilde{\HH} $.
These extensions are really canonical when the monodromy is unipotent. In
case it is not, the extensions depend on the choice of a logarithm,
but, as explained in \cite[\S~II~5]{dediff} this choice can be made
once and for all.  


Following \cite{ka}, \cite{ps} and \cite{sz} we call the variation
$(\hh,\ww_\bullet,\ff^\bullet,\boldsymbol{Q}_k)$ \emph{pre-admissible} if 
\begin{enumerate} 
\item the monodromy is quasi-unipotent,
\item the logarithm $N$ of the unipotent part of the monodromy has a weight filtration
  $M_\bullet(H,W_\bullet,N)$ relative to $W_\bullet$ on $H$, 
\item the subbundles $\ff^\bullet$ of $\hh$ extend to subbundles 
  $\tilde{\ff}^\bullet$ of $\tilde {\HH}$ in such a way that the
  coherent sheaves 
  $\Gr^{p}_{\tilde{\ff}}\Gr^{\tilde{\WW}}_{k}$ are locally free.
\end{enumerate}

Assume now that $X$ is an open submanifold of a manifold $\Xbar$,
where $D = \Xbar \setminus X$ is a normal crossings divisor. Let
$(\hh,\ww_\bullet,\ff^\bullet,\boldsymbol{Q}_k)$ be a graded-polarized
variation of mixed Hodge structures over the complex manifold $X$.  We
then call the variation
$(\hh,\ww_\bullet,\ff^\bullet,\boldsymbol{Q}_k)$ \emph{admissible} if
for every holomorphic map $\bar f\colon \Delta \to \Xbar$ with $\bar f(\Delta
^{\ast})\subset X$, the
variation $f^*\hh$ on $\Delta^*$ is pre-admissible. Here we denote by $f$
the restriction of $\bar{f}$ to $\Delta^*$.

In algebraic geometry, admissible variations of mixed Hodge structures
come about as follows. Let $\pi \colon Y \to X$ be a morphism of complex algebraic
varieties. Then there is a non-empty open subset $\iota \colon U\to X$  such
that the constructible sheaf $\hh=\iota^{\ast}\mathrm{R}^i \pi_*
\zz_Y$ is a local system that 
has a canonical structure
of admissible graded-polarized variation of mixed Hodge structures
$(\hh,\ww_\bullet,\ff^\bullet,\QQ_k)$.  Moreover, there is a 
finite \'etale map $g\colon \widetilde U\to U$ such that 
$g^{\ast}\hh$ has unipotent monodromy.

In general, the usual cohomological operations like direct images or
relative cohomology will produce \emph{mixed Hodge modules}
\cite{Saito1,Saito2} which are a 
generalization of the notion of admissible variations of polarizable
mixed Hodge structures, where the local system $\hh$ is replaced by a
perverse sheaf of $\qq$-vector spaces.
There is a criterion for when a mixed Hodge module is
indeed an admissible variation of mixed Hodge structures:
given a
mixed Hodge module $\hh$ (with a given polarization), if the
underlying perverse sheaf is a local system,
then $\hh$ is an admissible $\qq$-variation of polarized mixed Hodge
structures. See for instance \cite{Asakura} for a survey on mixed
Hodge modules.

For admissible variations of mixed Hodge structures we have the
following compatibility between the graded polarization and the
monodromy. Let $(H,W_\bullet,F^\bullet,Q_k)$ be a reference fiber of
the variation near the boundary divisor $D=\Xbar \setminus X$ of the
smooth algebraic variety $X$. We denote the local monodromy operators
around the branches of $D$ by $T_1,\ldots,T_m$, and the corresponding
logarithms of the unipotent part by $N_1,\ldots,N_m$. We denote by
$\mathfrak{g}_{\cc}$ 
the Lie algebra $\Lie G_{\cc}$ of the
group $C_{\cc}$ defined above. Then, in this generality, the $T_i$
belong to $G_{\cc}$, and the $N_i$ belong to $\mathfrak{g}_{\cc}$, for each
$i=1,\ldots,m$. The $\rr_{>0}$-span $\mathcal{C}$ of the local
monodromy logarithms $N_i$ inside $\mathfrak{g}_{\cc}$ is called the
\emph{open monodromy cone} of the reference fiber
$(H,W_\bullet,F^\bullet,Q_k)$. Each element of $\mathcal{C}$ is
nilpotent, and it can be proved that the relative weight filtration $M_\bullet$ of
$(H,W_\bullet)$ is constant on $\mathcal{C}$.\\

\noindent
\emph{The period map.}
To an admissible graded-polarized variation of mixed Hodge structures
$(\hh,\ww_\bullet,\ff^\bullet,\boldsymbol{Q}_k)$ over $X=(\Delta^*)^k
\times \Delta^{n-k}$ we can  associate a period map, as
follows. Let $\mm=\mm(h)$ be the period domain associated to
$(H,W_\bullet,Q_k)$ and set $G_{\cc}$ as above. Let
$\Gamma \subset G_{\cc}$ be the image of the monodromy representation
$\rho \colon \pi_1(X,x_0) \to G_{\cc}$. The period map $\phi \colon X
\to \Gamma \backslash \mm$ is the map that associates to $x \in X$
the Hodge filtration of $\hh_x$. The period map is holomorphic. 

Let $\mathbb{H} \subset \cc$ be the upper half plane. Let $e
\colon \mathbb{H}^k \to (\Delta^*)^k$ be the uniformization map given
by $(z_1,\ldots,z_k) \mapsto (\exp(2\pi i z_1),\ldots,\exp(2\pi i
z_k))$. Then along $e$ the period map $\phi$ lifts to a map
$\tilde{\phi} \colon \mathbb{H}^k \times \Delta^{n-k} \to \mm$. In
other words, we have the following commutative diagram
\[ \xymatrix{ \mathbb{H}^k \times \Delta^{n-k} \ar[r]^-{\tilde{\phi}} 
\ar[d]^{(e,\mathrm{id})} & \mm \ar[d] \\ (\Delta^*)^k \times \Delta^{n-k} \ar[r]^-\phi 
& \Gamma \setminus \mm   } 
\]
where the right hand arrow is the canonical projection.
As $N_i \in \Lie G_{\cc}$ we find $\exp(\sum_{i=1}^k z_iN_i) \in
G_{\cc}$ for all $z_1,\ldots,z_k \in \mathbb{H}$.  
Let $\tilde{\psi} \colon \mathbb{H}^k \times \Delta^{n-k} \to
\check{\mm}$ be the map given by 
\[ \tilde{\psi}(z_1,\ldots,z_k,q_{k+1},\ldots,q_n) = \exp(-\sum_{i=1}^k z_iN_i) . \tilde{\phi}(z_1,\ldots,z_k,q_{k+1},\ldots,q_n) \, . \] 
Then $\tilde{\psi}$ descends to an ``untwisted'' period map $\psi
\colon (\Delta^*)^k \times \Delta^{n-k} \to
\check{\mm}$, fitting in a commutative diagram
\[ \xymatrix{ \mathbb{H}^k \times \Delta^{n-k} \ar[r]^{\tilde{\psi}} 
\ar[d] & \check{\mm} \\ (\Delta^*)^k \times \Delta^{n-k} \ar[ur]^\psi 
&   } 
\]
Note that, importantly, the map $\psi$ takes values in the compact dual $\check{\mm}$, and not in a quotient of it.

\medskip
\noindent
\emph{Nilpotent orbit theorem.}
The following result is the starting point of G. Pearlstein's
Nilpotent orbit theorem 
(cf. \cite[Section~6]{pearlhiggs}) for admissible graded-polarized
variations of mixed Hodge structures and is enough for showing the
estimates we need. 

\begin{theorem} \label{nilpotentorbit} (G. Pearlstein) Let $(\hh,\ww_\bullet,\ff^\bullet,\boldsymbol{Q}_k)$ be an admissible graded-polarized variation of mixed Hodge structures over $X=(\Delta^*)^k \times \Delta^{n-k}$. Then the untwisted period map $\psi$ extends to a holomorphic map $\psi \colon \Delta^n \to \check{\mm}$. 
\end{theorem}

\subsection{Families of pointed polarized abelian varieties}	
\label{sec:families}
\

\medskip
\noindent
\emph{Period vectors.}
Assume that $(H,F^\bullet,Q)$ is a polarized pure Hodge structure of
weight $-1$, torsion-free, of type $(-1,0),(0,-1)$, and of rank $2g$.
Recall that given a
$\qq$-symplectic integral basis $(a_1,\ldots,a_g,b_1,\ldots,b_g)$ of
$(H,Q)$, there exists a unique
basis $(w_1,\ldots,w_g)$ of $F^0H_\cc$
determined by demanding that $w_i = -\sum_{j=1}^g \Omega_{ij}a_j +
b_i$ for some (period) matrix $\Omega \in M_g(\bb{C})$ (cf. equation
(\ref{eq:10})). We call this new basis the associated
\emph{normalized} basis. As we have seen, the Riemann bilinear
relations imply that $\Omega^t\Delta =\Delta \Omega$ and $\Delta \Im
\Omega >0$. 

Assume an extension
\begin{equation} \label{anyextension} 0 \to H \to H' \to \zz(0) \to 0 
\end{equation}
in the category of mixed Hodge structures is given. Then $H'$ has
weight filtration
\[ W_\bullet \colon \quad 0 \subset W_{-1} = H_\qq \subset W_0 = H_\qq' \, . \]
Taking $F^0(-)_\cc$ in (\ref{anyextension}) yields the extension
\[ 0 \to F^0H_\cc \to F^0H'_\cc \to \cc \to 0 \] of $\cc$-vector
spaces. As can be readily checked, for each $a_0 \in H'$ that lifts
the canonical generator of $\zz(0)$ in (\ref{anyextension}) there
exists a unique $w_0 \in F^0H'_\cc$ such that
$w_0 \in a_0 + \cc$-$\spann(a_1,\ldots,a_g)$. Given such a lift $a_0$,
we let $\delta_{H'} = (\delta_1,\ldots,\delta_g)^{t} \in
\on{Col}_{g}(\bb{C})$ be the 
coordinate vector determined by the identity
$w_0=a_0+\sum_{j=1}^g \delta_j a_j$. We call $\delta_{H'}$ the \emph{period
vector} of the mixed Hodge structure $(H',F^\bullet,W_\bullet)$ on the
basis $(a_0,a_1,\ldots,a_g,b_1,\ldots,b_g)$ of the $\zz$-module
$H'$. It can be verified that replacing $a_0$ by some element from
$a_0+H$ changes $\delta$ by an element of $\zz^g+\Omega \zz^g$. The
resulting map
$\mathrm{Ext}^1_{\mathrm{MHS}}(\zz(0),H) \to \cc^g/(\zz^g+\Omega
\zz^g)$ is finite, and gives
$\mathrm{Ext}^1_{\mathrm{MHS}}(\zz(0),H)$ a canonical structure of
complex torus.

Let $A=J(H)$ be the Jacobian of $H$. Thus $A$ is a polarized complex
abelian variety of dimension $g$ with 
$H=\rmH_1(A)$. Let $\nu \in A$,
and write $H(\nu)$ for the 
relative homology group $\rmH_1(A,\{0,\nu\})$. There is an
extension of mixed Hodge structures
\[ 0 \to H \to H(\nu) \to \zz(0) \to 0 \]
canonically associated to $(A,\nu)$. Here $\zz(0)$ is to be identified
with the reduced homology group $\tilde{\rmH}_0(\{0,\nu \})$. 
The map $A \to \mathrm{Ext}^1_{\mathrm{MHS}}(\zz(0),H)$ given by
sending $\nu$ to the extension $H(\nu)$ is a bijection, compatible
with the structure of complex torus on left and right hand side.

\medskip
\noindent
\emph{The period map of a family of pointed polarized
  abelian varieties.} 
Let $X$ be a smooth complex variety with an open immersion $X\subset
\Xbar$  into a smooth complex algebraic variety, with $D = \Xbar
\setminus X$ a normal crossings divisor.  
Let $(\hh,\ff^\bullet,\boldsymbol{Q}_k)$ be a variation of polarized torsion-free pure Hodge
structures of weight $-1$ and type $(-1,0),(0,-1)$ over $X$. We note that such
a pure polarized variation is necessarily admissible. Write $Y=J(\hh)$ and let $\pi \colon Y \to X$  be the 
associated analytic family of polarized abelian
varieties over $X$, with polarization $\lambda \colon Y \to Y^\lor$.

We will now work locally complex analytically. Thus we will suppose
that $\Xbar$ is the polydisk $\Delta^n$, and $D$ is the divisor given by the
 equation $q_1\cdots q_k=0$, so that $X = (\Delta^*)^k \times
 \Delta^{n-k}$. 
We further assume that all local monodromy operators $T_1,\ldots,T_k$ about
the various branches determined by $q_1,\ldots,q_k$ are unipotent (for
instance, this is the case if the family extends as a semiabelian scheme
$\Ybar \to \Xbar$). Let $g$ be the
relative dimension of $Y \to X$.  Also, we will henceforth usually suppress the polarization
from our notation.  

Let $(H,F^\bullet)$ be a reference fiber of $\hh$ near the origin. Let
$N$ be any element of the open monodromy cone of $H$. Then we have
$N^2=0$ and the
filtration associated to $N$ simply reads
\[ 0 \subset M_{-2} \subset M_{-1} \subset M_0 = H_\qq \]
with $M_{-2}=\Im N $ and $M_{-1}=\Ker N $. Since, in this case the
group $G_{\rr}$ acts transitively on $\mm$, the operator $N$ belongs to the
Lie algebra of $G_{\rr}$, thus there exist a $\qq$-symplectic integral basis
$(a_1,\ldots,a_g,b_1,\ldots,b_g)$ of
$(H,Q)$ and a non-negative integer $r \leq g$ such that:
\begin{enumerate}
\item $M_{-2}=\spann{(a_1,\ldots,a_r)}$, 
\item $M_{-1}=\spann{(a_1,\ldots,a_g,b_{r+1},\ldots,b_g)}$. 
\end{enumerate}
In particular,
$(\bar{a}_{r+1},\ldots,\bar{a}_g,\bar{b}_{r+1},\ldots,\bar{b}_g)$ is a
$\qq$-symplectic integral basis of the pure polarized Hodge structure
$\Gr_{-1}^M H$ of type $(-1,0), (0,-1)$. Clearly, with respect to this
basis, each local monodromy operator $N_j$ has the form 
\[ N_j = \begin{pmatrix}[c|c] 0 & A'_j \\
\hline 
0 & 0 \\
\end{pmatrix} \, . \]
Let $\Delta $ be the matrix associated to the polarization as in
equation \eqref{type}.
Each $A'_j$ is integral and the $g$-by-$g$ matrices $A_{j}\defeq \Delta A'_{j}$
are symmetric and
positive semidefinite. Moreover, the left upper $r$-by-$r$
block of $A_j$ is positive definite.  See \cite{cat} for more details in the above
construction. 

To avoid the appearance of the polarization
matrix $\Delta $ and thus to simplify the notation we will sometimes replace the $\qq$-symplectic
integral basis
$(a,b)$ by the symplectic $\qq$-basis $(a\Delta ^{-1},b)$. In this new
basis each local monodromy operator $N_j$ has the form 
\[ N_j = \begin{pmatrix}[c|c] 0 & A_j \\
\hline 
0 & 0 \\
\end{pmatrix} \, . \]

On this new basis we can realize the period domain
associated to $H$ as the usual Siegel's upper
half space $\mathbb{H}_g$ of rank $g$. We have $G_{\rr}=\Sp(2g,\rr)$, and the
action on $\mathbb{H}_g$ is given by the usual prescription 
\[ \begin{pmatrix}[c|c] A & B \\
\hline 
C & D \\
\end{pmatrix} \cdot M = (AM+B)(CM+D)^{-1} \, , \,     
\begin{pmatrix}[c|c] A & B \\
\hline 
C & D \\
\end{pmatrix} \in \Sp(2g,\rr) \, , \, M \in \mathbb{H}_g \, . \]

In this representation the period map $\Omega \colon X \to \Gamma
\setminus \mathbb{H}_g$ is 
made explicit by associating to each $x \in X$ the matrix $\Omega
(x)=\Delta\Omega _{Y_{x}} $, where $\Omega _{Y_{x}}$ is the period
matrix of the fiber $Y_x$ 
on the chosen $\qq$-symplectic integral basis of $H$. Here $\Gamma$ is the
image of the monodromy representation into
$\Sp(2g,\rr)$. In the new basis, the monodromy
representation sends the
local monodromy operator $T_{j}$ to the matrix
\[ \begin{pmatrix}[c|c] 1 & A_j \\
\hline 
0 & 1 \\
\end{pmatrix} \in \Sp(2g,\rr)\, . \]

We will now extend this picture to include an admissible normal function $\nu$.  
\begin{definition}\label{def:3} Let $U\subset X$ be a non-empty open subset. An
  \emph{admissible normal function} of $Y$ on $U$ is a holomorphic section $\nu \colon U \to Y|_U$ such that the associated graded-polarized mixed Hodge structures $\hh_{x}(\nu )$, $x\in U$ form an admissible variation over $U$.
\end{definition}
Let $\nu$ be an admissible normal function of $Y$ on $X$. Varying $x \in X$ we thus find an extension
\[ 0 \to \hh \to \hh(\nu) \to \zz(0) \to 0 \]
of admissible variations of graded-polarized mixed Hodge structure. The weight filtration of this
variation looks like
\[ W_\bullet \colon \quad 0 \subset \boldsymbol{W}_{-1}=\boldsymbol{H}_\qq \subset \boldsymbol{W}_0 = \boldsymbol{H}(\nu)_\qq \, , \]
so that $\Gr_{-1}^{\boldsymbol{W}} \boldsymbol{H}(\nu)_\qq=\boldsymbol{H}_\qq$ and $\Gr_{0}^{\boldsymbol{W}} \boldsymbol{H}(\nu)=\qq(0)$. We denote the
Hodge filtration of $\hh(\nu)$ by $\ff^\bullet$. We start by taking a
reference fiber $H(\nu)$ of $\boldsymbol{H}(\nu)$ and augmenting our 
chosen $\qq$-symplectic integral basis of $H$ by an $a_0 \in H(\nu)$ lifting the
canonical generator of $\zz(0)$. 

Since by assumption,  $\hh(\nu)$ is an admissible variation of polarized
mixed Hodge structures, the relative weight filtration
$M'_\bullet$ on our reference fiber $H(\nu)$ exists. Let $N'$ be an
element of the open monodromy cone of $H(\nu)$ such that $N=N'|_{H}$. We
will now proceed to determine the matrix shape of $N'$ on the
basis $(a_0,a\Delta ^{-1},b)$ of $H(\nu)$. As $N'^2=0$,
the filtration associated to $N'$ on $H(\nu)$ is 
\[ L_\bullet \colon \quad 0 \subset L_{-1} \subset L_0 \subset
  L_1=H(\nu)_\qq \, , \]
with $L_{-1}=\Im(N')$, $L_0=\Ker(N')$. As the monodromy action on
$\Gr_0^W =\qq(0)$ is trivial, we have that $\Im(N') \subset H_\qq$, so
that $N'^{-1}H_\qq=H(\nu)_\qq$ (here by $N'^{-1}H_\qq$ we denote the
inverse image of $H_\qq$ under $N'$). As $W_\bullet$ has length two, and as by
admissibility the weight filtration of $N'$ relative to $W_\bullet$
exists, as we noted above it follows that $N'$ is strict. Explicitly,
we have that $H(\nu)_\qq=N'^{-1}H_\qq=H_\qq+\Ker(N')$. The equality
$H(\nu)_\qq=H_\qq+\Ker(N')$ implies that $\Ker(N') \supsetneqq
\Ker(N)$ and hence that $\Im(N')=\Im(N)$.

The period domain associated to $(H(\nu),W_\bullet)$ can be realized
as $\cc^g \times \mathbb{H}_g$. The group $G_{\rr}$ in this case is
\[ G_{\rr} = \left\{ \begin{pmatrix}[c|c|c]
1 & 0 & 0 \\ \hline
m & A & B \\ \hline
n & C & D
\end{pmatrix}  \, :  \, m, n \in \rr^g \, , \, \begin{pmatrix}[c|c]
A & B \\ \hline C & D \end{pmatrix} \in \Sp(2g,\rr)
\right\} \, .  \]
The action of $G_\rr$ on $\cc^g \times \mathbb{H}_g$ is given by
\[ \begin{pmatrix}[c|c|c]
1 & 0 & 0 \\ \hline
m & A & B \\ \hline
n & C & D
\end{pmatrix} (v,M) = (v+m+Mn,(AM+B)(CM+D)^{-1}) \, , \, v \in \cc^g \, , \, M \in \mathbb{H}_g \, .
\]

Varying $x \in X=(\Delta^*)^k \times
 \Delta^{n-k}$ and then taking $F^0$ we obtain a period map
associated to the variation $\hh(\nu)$ 
\[ (\delta,\Omega) \colon X \to \Gamma \setminus (\cc^g \times
  \mathbb{H}_g) \] 
that is given by
\[
(\delta (x),\Omega (x))=(\Delta \delta _{H(\nu (x))}, \Delta \Omega _{Y_{x}}).
\]
We denote by
\[ (\tilde \delta,\tilde \Omega) \colon \bb{H}^
k\times \Delta ^{n-k} \to \cc^g \times \mathbb{H}_g \]
the lift of the period map along the map $(e,\mathrm{id}) \colon \mathbb{H}^k \times \Delta^{n-k} \to X$, where we recall that we denote by $e\colon \bb{H}^{k}\to (\Delta
^{\ast})^{k} $ the map
\begin{displaymath}
  e(z_{1},\dots,z_{k})=(\exp(2\pi i z_{1}),\dots,\exp(2\pi i z_{k})).
\end{displaymath}
\begin{theorem} \label{asympt} 
There exist a holomorphic map $\psi \colon \Delta^n \to S_g(\cc)$, a
holomorphic map $\alpha \colon \Delta^n \to \cc^g$,  
and vectors $c_1,\ldots, c_k \in \qq^g$ with $\Delta ^{-1}A_jc_j \in
\zz^g$ for $j=1,\ldots,k$ such that for $(z,t) \in \mathbb{H}^k \times
\Delta^{n-k}$ with $e(z)$ sufficiently close to zero the equalities
\[ \tilde{\Omega}(z,t) = \sum_{j=1}^k z_j A_j  + \psi(e(z),t) \, ,
  \quad \tilde{\delta}(z,t) = \sum_{j=1}^k z_j A_jc_j  + \alpha(e(z),t)
\] 
hold in $S_g(\cc)$ resp.\ $\cc^g$. 
\end{theorem}
\begin{proof} Let $N_j$ denote the local monodromy operator of $H$
  around the branch of $D$ determined by $q_j=0$. We have
\[ \exp(z_jN_j)=T_j^{z_j} = \begin{pmatrix}[c|c] 1 & z_jA_j \\
\hline 
0 & 1 \\
\end{pmatrix} \]
and hence
  $\exp(z_jN_j).M=z_jA_j+M$ for each $M \in \mathbb{H}_g$, $z_j
  \in U$, and 
  $j=1,\ldots,k$ (here $U$ is an open subset of $\bb{H}$ consisting of
  points with sufficiently large imaginary part). Denote by
  $\bb{P}_{g}$ the compact dual of 
  $\bb{H}_{g}$. The untwisted period map $\psi \colon
  \Delta^n \to \mathbb{P}_g$ obtained by Theorem \ref{nilpotentorbit}
  extending $\exp(-\sum_{j=1}^k
  z_jN_j).\tilde{\Omega}(z,t)$ factors through $S_g(\cc)\subset
  \mathbb{P}_g$. We obtain the equalities
\[ \tilde{\Omega}(z,t) = \exp(\sum_{j=1}^k z_jN_j).\psi(e(z),t) =
\sum_{j=1}^k z_j A_j + \psi(e(z),t) \]
in $S_g(\cc)$.

Let $N'_j$ denote the local monodromy operator of $H(\nu)$ around the
branch of $D$ determined by $q_j=0$. The equality
$\Im(N'_j)=\Im(N'_j|_{H_\qq})$ on $H(\nu)_\qq$ that follows from our
above considerations
shows that $N'_j$ has a matrix
\[  \begin{pmatrix}[c|c|c]
0 & 0 & 0 \\ \hline
\Delta ^{-1}A_jc_j & 0 & \Delta ^{-1}A_j \\ \hline
0 & 0 & 0
\end{pmatrix} \]
on the integral basis $(a_0,a_1,\ldots,a_g,b_1,\ldots,b_g)$, for some
$c_j \in \qq^g$. Since the monodromy is integral in such basis, we
deduce that $\Delta ^{-1}A_jc_j$ has to be integral. In the $\qq$-basis
$(a_{0},a\Delta ^{-1},b)$, the matrix of   $N'_j$ is
\[  \begin{pmatrix}[c|c|c]
0 & 0 & 0 \\ \hline
A_jc_j & 0 & A_j \\ \hline
0 & 0 & 0
\end{pmatrix} \]

Then for $(v,M) \in \cc^g \times \mathbb{H}_g$ and $z_j \in U$
we have $\exp(z_jN'_j).(v,M)=(v+z_jA_jc_j,M+z_jA_j)$. Let
$(\alpha,\psi) \colon \Delta^n \to \cc^g \times \mathbb{P}_g$ denote
the untwisted period map. We find the equalities
\[ \tilde{\delta}(z,t) = \exp(\sum_{j=1}^k z_jN_j').\alpha(e(z),t) =
  \sum_{j=1}^k z_j A_jc_j + \alpha(e(z),t)   \] 
in $\cc^g$.
\end{proof}

\medskip
\noindent
\emph{The norm of a section.}
Denote by $\pp$ the
Poincar\'e bundle on $Y \times_X Y^\lor$ with its canonical $C^\infty$
hermitian metric as described in Section
\ref{sec:poincare-bundle}. Given two admissible normal functions $\nu
,\mu \colon X \to Y$ we will denote
\begin{displaymath}
  \pp_{\nu ,\mu }=(\nu ,\lambda \mu )^{\ast}\pp,\qquad \pp_{\nu
  }=\pp_{\nu ,\nu }, 
\end{displaymath}
where $\lambda \colon Y\to Y^{\lor}$ is the isogeny provided by the
polarization. We are interested in studying the singularities of the
metric of $\pp_{\nu ,\mu }$ when we approach the boundary of $X$. 

Consider the five maps
\begin{displaymath}
    m, \, p_{1,3}, \, p_{1,4}, \, p_{2,3}, \, p_{2,4} \colon Y \times_X Y \times_X
    Y^\lor \times_X Y^\lor \longrightarrow Y \times_X
    Y^\lor \, ,
\end{displaymath}
where $m(x,y,z,t)=(x+y,z+t)$ and $p_{i,j}$ is the projection over
the factors $i,j$. Then we have a canonical isomorphism
  \begin{equation} \label{biextproperty}
    m^{\ast}\pp   \isom p_{1,3}^{\ast}\pp \otimes p_{1,4}^{\ast}\pp
    \otimes p_{2,3}^{\ast}\pp \otimes p_{2,4}^{\ast}\pp 
  \end{equation}
of holomorphic line bundles over $Y \times_X Y \times_X
    Y^\lor \times_X Y^\lor$, in other words, the Poincar\'e bundle is
    a biextension on $Y \times_X Y^\lor$ in the sense of
    \cite[Expos\'e VII]{SGA7.1}. The explicit description of
    the cocycle $b_{\pp}$ in equation \eqref{eq:5} and of the metric
    of the Poincar\'e bundle in Remark \ref{formulanorm} shows that
    the canonical isomorphism (\ref{biextproperty}) is in fact an
    isometry for the canonical induced metrics on left and right hand
    side. We obtain in particular 
\begin{lemma}\label{lemm:1}
Let $\nu_1, \nu_2, \mu_1, \mu_2$ be holomorphic sections of the family
$Y \to X$. Then we have a canonical isometry
\begin{displaymath}
  (\nu_1+\nu_2,\lambda(\mu_1+\mu_2))^*\pp \isom
(\nu_1,\lambda \mu_1)^*\pp \otimes 
(\nu_1,\lambda \mu_2)^*\pp \otimes 
(\nu_2,\lambda \mu_1)^*\pp \otimes 
(\nu_2,\lambda \mu_2)^*\pp 
\end{displaymath}
of hermitian line bundles on $X$. 
\end{lemma}
Let $f \colon Y \times_X Y^\lor \to Y \times_X Y^\lor$ be the map given by $(x,\ell) \mapsto (\lambda^\lor(\ell),\lambda(x))$. Then we have a canonical isometry $f^*\ca{P} \isom \ca{P}^{\otimes d}$, where $d=\det \Delta$ is the degree of the polarization $\lambda$. This leads to canonical isometries
\[  \ca{P}_{\nu,\mu}^{\otimes d} \isom \ca{P}_{\mu,\nu}^{\otimes d}  \quad \textrm{and} \quad
\ca{P}_{\nu,\mu}^{\otimes 2d} \isom \ca{P}_{\nu+\mu}^{\otimes d} \otimes \ca{P}_\nu^{\otimes (-d)} \otimes \ca{P}_\mu^{\otimes (-d)} \, . \]
Hence, in order to study the singularities of
the metric on
$\pp_{\nu,\mu}$ it suffices to study the singularities of the metric
on $\pp_\nu$, $\pp_\mu$ and $\pp_{\nu+\mu}$. In particular,  for the
purpose of proving our main results, it suffices to focus on the
diagonal cases $\pp_\nu$.

Let $\nu$ be an admissible normal function 
of the family $Y \to X$. Let $s$ be a section on $X$ of $\pp_{\nu }$.
Interpreting the Poincar\'e bundle
as parametrizing biextension variations of mixed Hodge structures (see
\cite{hainbiext}) we can canonically associate to the section $s$  a biextension variation of
mixed Hodge structures over $X$. We say that the section $s$ is
\emph{admissible} if this variation is admissible.

 Let $x_0$ be a
point of $\Xbar$. The purpose of the present section is to give an asymptotic
expansion of the logarithm of the norm of an admissible section of
$\pp_\nu$ near $x_0$. From equation
\eqref{eq:23} it follows that it suffices to give asymptotic
expansions of the period matrix of the family $Y \to X$, and of the
period vector (see below) associated to $\nu$.

We use now Theorem \ref{asympt} to
obtain an expression of the norm of the section $s$. 
Let $\aabs{-}$ denote the canonical metric on
$\ca{P}_\nu=\nu^*\ca{P}$. Continuing the notation from Theorem \ref{asympt}, 
let $a = 2\pi \Im \alpha$ and $B=2\pi \Im \psi$. For
$j=1,\ldots,k$ let $x_j = -\log|t_j|$.
\begin{corollary} \label{explicitnorm} 
For every admissible section $s$ of $\ca{P}_\nu$ on $ (\Delta^*)^
{k} \times \Delta^{n-k}$ there exists a meromorphic function $h$ on
$\Delta^n$ which is holomorphic on $ (\Delta^*)^  {k} \times
\Delta^{n-k}$, such that the
identity 
\begin{equation}\label{eq:def_of_f_eta}
-\log\aabs{s} = -\log\abs{h} + \left(\sum_{j=1}^kx_jA_jc_j + a
\right)^t \left(\sum_{j=1}^kx_jA_j + B \right)^{-1}
\left(\sum_{j=1}^kx_jA_jc_j + a \right)
\end{equation}
holds on $ (\Delta^*)^{k} \times \Delta^{n-k}$. 
\end{corollary}
\begin{proof} The vector $z$ and the matrix $\Omega $ in Theorem
  \ref{explicitmetricPoinc} are expressed in the integral basis
  $(a,b)$, while $\delta (x)$ and $\Omega (x)$ are expressed in the
  $\qq$-basis $(a\Delta ^{-1},b)$. Writing $z=\Delta ^{-1}\delta (x)$
  and $\Omega =\Delta ^{-1}\Omega (x)$, we obtain that 
\[ -\log\aabs{s(x)} = -\log\abs{h(x)} + 2\pi (\Im \delta(x))^t (\Im
  \Omega(x))^{-1} (\Im \delta(x))   \]
for a suitable meromorphic function $h$ on $\Delta^n$ which is
holomorphic on $(\Delta^*)^  {k} \times \Delta^{n-k}$. 
Note that, even though $\Omega(x), \delta(x)$ are multivalued, their
imaginary parts are single valued. From Theorem \ref{asympt} we
obtain, noting that $\Im z_j = -\frac{1}{2\pi}\log|t_j|$,
\[  
\Im \Omega (x)= -\frac{1}{2\pi} \sum_{j=1}^k A_j \log|t_j| + \Im
\psi \, , \quad 
\Im \delta (x)= -\frac{1}{2\pi} \sum_{j=1}^k A_jc_j \log|t_j| + \Im \alpha
\, .  
\]
Combining we find equation (\ref{eq:def_of_f_eta}).
\end{proof}

\section{Normlike functions}
 \label{technical}

\newcommand{\tmatrix}{A}
\newcommand{\tvector}{c}
\newcommand{\tmatfun}{B}
\newcommand{\tvectfun}{a}

The purpose of this section is to carry out a systematic study of the functions
\[ \varphi = \left(\sum_{j=1}^kx_jA_jc_j + a \right)^t \left(\sum_{j=1}^kx_jA_j + B \right)^{-1} \left(\sum_{j=1}^kx_jA_jc_j + a \right)   \] 
that appear on the right hand side of the equality in Corollary
\ref{explicitnorm}. We call such functions \emph{normlike}
functions.

Let $f\colon \rr^{n}\to \rr$ be a function. The \emph{recession function} of
$f$ is the function
\begin{displaymath}
  \on{rec}(f)(x)=\lim_{\lambda \to +\infty}\frac{1}{\lambda }f(\lambda  x).
\end{displaymath}
If the recession function exists, it is homogeneous of weight one, in
the sense that
\begin{displaymath}
  \on{rec}(f)(\mu x)=\mu \on{rec}(f)(x), \text{ for } \mu\in \rr_{>0}.
\end{displaymath}

We show that normlike functions $\varphi$ have a well-defined recession
function $\on{rec}(\varphi)$ with respect to the variables $x_j$, and
we are able to calculate $\on{rec} (\varphi)$ explicitly. In our main
technical lemma Theorem
\ref{thm:main_technical} we give bounds for the difference $\varphi -
\on{rec} (\varphi)$ and, in the case where $k=1$, for the first and
second order derivatives of $\varphi - \on{rec} (\varphi)$. The bound on
the difference will be key to the proof of our first main result
Theorem \ref{singbiext}, the bounds on the derivatives will be used in
our proof of Theorem \ref{localint}. In section
\ref{propertiesnormlike} we prove, among other things, that the
recession functions $\on{rec}( \varphi)$ are convex. This will lead to
the effectivity statement in Theorem \ref{theorem:effectivity}.

\subsection{Some definitions} 

Recall that we have denoted by $M_{r}(\bb{R})$
the space of $r$-by-$r$ matrices with real coefficients, by
$S_r^+(\bb{R})\subset M_{r}(\bb{R})$ the cone of 
symmetric positive semidefinite real matrices inside $M_r(\bb{R})$,
and by $S_r^{++}(\bb{R})\subset S^+_r(\bb{R})$ the cone of symmetric
positive definite real matrices.

\begin{lemma} \label{lem:simultaneous} Let $N_{1},\dots,N_{k}$ be a
  finite set of positive semidefinite symmetric real $g$-by-$g$ matrices such that $N_{1}+\dots+N_{k}$ has rank $r$. Then there
  exists an orthogonal matrix $u \in O_g(\bb{R})$ such that, upon writing $M_{i}=u^tN_iu$ for $i=1,\dots,k$, we have
\begin{equation*}
M_i = \begin{pmatrix}[c|c] M_i' & 0\\
\hline 
0 & 0 \\
\end{pmatrix}, 
\end{equation*}
with all $M_i'\in S_{r}^+(\bb{R})$ and $\sum M_i'\in S_{r}^{++}(\bb{R})$.
\end{lemma}
\begin{proof} It will be convenient to use the language of bilinear forms. If $Q$ is a symmetric positive semidefinite bilinear
  form on $\bb{R}^g$ and $f_1,\ldots, f_g$ is a basis of $\bb{R}^g$
  such that $Q(f_\alpha ,f_\alpha )=0$ for $\alpha =r+1,\ldots,g$,
  then $Q(f_\alpha ,f_\beta )=0$ 
  for $\beta =1,\ldots,g$ and $\alpha =r+1,\ldots,g$. Indeed, for all $\lambda
  \in \bb{R}$ we have  $Q(\lambda f_\alpha -f_\beta ,\lambda f_\alpha  - f_\beta ) \geq 0$,
  that is \[ -2\lambda Q(f_\alpha ,f_\beta ) + Q(f_\beta ,f_\beta ) \geq 0 \, . \]
  Since this inequality is satisfied for all $\lambda$ we
  deduce that $Q(f_\alpha ,f_\beta )=0$.

Let $N=N_1+\cdots+N_k$, and denote by $Q$ the symmetric positive semidefinite bilinear form that $N$ defines on the standard basis $(e_1,\ldots,e_g)$ of $\bb{R}^g$. Note that $Q$ has rank $r$. By the spectral theorem, upon replacing the basis $(e_1,\ldots,e_g)$ of $\bb{R}^g$ by $(f_1,\ldots,f_g)=(e_1,\ldots,e_g)u$ for some orthogonal matrix $u$ we can assume that 
the expression of $Q$ in the basis $(f_1,\ldots,f_g)$ is  
\begin{equation*}
  M= \begin{pmatrix}[c|c] A & 0\\
    \hline 
    0 & 0 \\
\end{pmatrix}, 
\end{equation*}
with $A\in S^{+}_{r}(\bb{R})$ invertible and diagonal. In particular,
$Q(f_\alpha ,f_\alpha )=0$ for $\alpha =r+1,\dots, g$. For $i=1,\ldots,k$ let $Q_i$ denote the symmetric positive semi-definite bilinear form that $N_i$ defines on the standard basis $(e_1,\ldots,e_g)$ of $\bb{R}^g$. Note that $Q=Q_1+\cdots+Q_k$. Since all the $Q_{i}$ are
positive semidefinite, we deduce that $Q_{i}(f_\alpha ,f_\alpha )=0$ for
$i=1,\dots ,k$. 
Note that $M_i=u^tN_iu$ is the
expression of $Q_{i}$ in the basis $(f_1,\ldots,f_g)$. By the previous discussion we have
\[ M_i= \begin{pmatrix}[c|c] M_i' & 0 \\
\hline 
0 & 0 \\
\end{pmatrix} \, ,\]
with $M'_{i}\in S_{r}^+(\bb{R})$ and $\sum M_i'=A\in
S_{r}^{++}(\bb{R})$, proving the lemma.
\end{proof}

Suppose we are given the following data:
\begin{itemize}
\item[-] three integers $k \ge 0$, $m \ge 0$, $g \ge 0$;
\item[-] a real number $\kappa \ge 0$;
\item[-] a compact subset $K \sub \bb{R}^m$;
\item[-] matrices \label{item:1} $\tmatrix_1, \ldots, \tmatrix_k \in S_g^+(\bb{R})$ all
  of rank $\ge 1$;
\item[-] vectors $\tvector_1, \ldots, \tvector_k \in \bb{R}^g$;
\item[-] functions $\tvectfun\colon K \ra \bb{R}^g$ and $\tmatfun\colon K \ra
  S_g(\bb{R})$ which are  restrictions of smooth functions on some open
  neighbourhood of $K$; 
\end{itemize}
such that for all $(x_1,\ldots,x_k, \lambda) \in \bb{R}_{>\kappa}^k
\times K$, we have that 
\begin{equation}\label{equation:positivity}
P(x_1,\ldots,x_k,\lambda) \defeq \sum_{i=1}^kx_i\tmatrix_i + \tmatfun(\lambda) >0 \, . 
\end{equation}
Note that if $g=0$, then necessarily $k=0$.

To these data we associate a smooth function $\varphi\colon
\bb{R}_{>\kappa}^k\times K 
\ra \bb{R}$ by
\begin{multline}
  \varphi(x_1, \ldots, x_k, \lambda) =\\
\left( \sum_{i=1}^kx_i\tmatrix_i\tvector_i + \tvectfun(\lambda) \right)^t 
\left( \sum_{i=1}^kx_i\tmatrix_i + \tmatfun(\lambda) \right)^{-1}
\left( \sum_{i=1}^kx_i\tmatrix_i\tvector_i + \tvectfun(\lambda) \right) .
\end{multline}
By condition (\ref{equation:positivity}),  the function $\varphi$ is
well-defined and its values are non-negative. We call $\varphi$ the
\emph{normlike} function associated to the $4$-tuple
$((\tmatrix_i), (\tvector_i), \tvectfun, \tmatfun)$. We call the natural number $k$ the
\emph{dimension} of $\varphi$.  Write
$r = \on{rk} \sum_{i=1}^k x_i \tmatrix_i$ for some (hence all)
$(x_1, \ldots, x_k) \in \bb{R}^k_{>\kappa}$. Note that
$r \geq 1$ if $k >0$. 

Let $u \in O_g(\bb{R})$. Replacing the vector $\tvector_i$ by
$u^{-1}\tvector_i$, $\tvectfun$ by $u^{-1}\tvectfun$, 
the matrix $\tmatfun$ by $u^{t}\tmatfun u$ and $\tmatrix_i$ by $u^{t}\tmatrix_i u$ one checks that the function $\varphi$ remains unchanged.
By Lemma \ref{lem:simultaneous} we can thus restrict to considering normlike functions where the $A_i$ have the shape
\begin{equation*}
\tmatrix_i = \begin{pmatrix}[c|c] \tmatrix_i'& 0_{r,g-r}\\
\hline 
0_{g-r,r} & 0_{g-r,g-r}\\
\end{pmatrix}, 
\end{equation*}
with each $\tmatrix'_i \in S_r^+(\bb{R})$ and such that $\sum x_i
\tmatrix'_i\in S^{++}_r(\bb{R})$ for all $(x_1,\ldots,x_k) \in
\bb{R}_{>\kappa}^k$ (hence for all 
$(x_1,\ldots,x_k)\in\bb{R}_{>0}^k$).

From now on we assume that the matrices $A_i$ indeed have this shape.
We write
\begin{equation*}
\tvector_i = \mat{\tvector_i'\\
\hline
\star_{g-r}}, \;\;\; 
\tvectfun = \mat{\tvectfun_1\\
\hline
\tvectfun_2}, \;\;\; 
\text{and} \;\;\; \tmatfun =  \begin{pmatrix}[c|c] \tmatfun_{11}& \tmatfun_{12}\\
\hline 
\tmatfun_{21} & \tmatfun_{22}\\
\end{pmatrix}
\end{equation*}
where $c_i'$ and $a_1$ have size $r$, and $B_{11}$ is an $r$-by-$r$ matrix. 
The second block of the vector $\tvector_{i}$ is marked with an asterisk
because the function $\varphi$ is independent of its value.
 Condition \ref{equation:positivity} implies that
$\tmatfun_{22}(\lambda)$ is positive definite for all
$\lambda \in K$, and the symmetry of $\tmatfun$ implies that $\tmatfun_{21}=\tmatfun_{12}^t$.

We define another smooth function $  f \colon
\bb{R}_{>\kappa}^k\times K \ra \bb{R}$ by 
\begin{equation} \label{recessionexplicit} 
  f(x_1, \ldots, x_k, \lambda) =
  \left(\sum_{i=1}^kx_i\tmatrix'_i\tvector'_i\right)^t
  \left(\sum_{i=1}^kx_i\tmatrix'_i\right)^{-1}
  \left(\sum_{i=1}^kx_i\tmatrix'_i\tvector'_i\right). 
\end{equation}
This function $f$ is well defined as $\sum_{i=1}^kx_i\tmatrix'_i$ is positive definite on $\bb{R}_{> 0}^k$.
The function $f$ depends trivially on $\lambda$ and is clearly
homogeneous of degree 1 in the $x_i$, and so defines a smooth
function $\bb{R}_{> 0}^k\ra \bb{R}$, which we also call $f$. Again, the values of $f$ are non-negative. By convention, if $k=0$, the function $f$ is zero.

Finally, the ``recession'' of $\varphi$ is defined as the pointwise limit
\begin{equation*}
\begin{matrix}
\on{rec}(\varphi)\colon & \bb{R}_{>\kappa}^k\times K & \ra &\bb{R}\\
&(x_1, \ldots, x_k, \lambda) & \mapsto &\on{lim}_{\mu \ra \infty}
\frac{1}{\mu} \varphi(\mu x_1, \ldots, \mu x_k, \lambda),  
\end{matrix}
\end{equation*}
if it exists. Again, if $k=0$, then $\on{rec}(\varphi)=0$. 

\subsection{Statement of the technical lemma}

We can now state the ``main technical lemma'':
\begin{theorem} \label{thm:main_technical}
In the notation of the previous section, write
$\varphi_{0}=\varphi-f$. Note that $\varphi_0$ is a smooth function on $\bb{R}_{>\kappa} \times K$. Then
\begin{enumerate}
\item \label{main_technical_bounddiff} the function $\abs{\varphi_{0}}$ is
  bounded on $\bb{R}_{>\kappa'}^k \times K$ for some $\kappa' \ge
  \kappa$. The recession of $\varphi $ exists and is equal to $f$. In
  particular, $\on{rec}( \varphi)$ is independent of the
  parameter~$\lambda$;
\item \label{main_technical_fbound} the function $f$ is bounded on the
  open simplex $\Delta^0 = \{(x_1, \ldots, x_k) \in \bb{R}_{>0}^k :
  \sum_{i=1}^k x_i = 1\}$;
\item \label{main_technical_k=1} when $k=1$, 
\begin{enumerate}
\item the function $\varphi_0 \colon \bb{R}_{>\kappa} \times K \ra \bb{R}$
  extends continuously to a function from $\ov{\bb{R}_{>\kappa}}
  \times K$ to $\bb{R}$, where by $\ov{\bb{R}_{>\kappa}}$ we denote
  $\bb{R}_{>\kappa} \sqcup \{\infty\}$ with the natural topology; 
\item the derivatives of $\varphi_{0}$ satisfy the estimates
\begin{equation*}
\frac{\partial \varphi_0}{\partial x_1} = O(x_1^{-2}) \quad \textrm{and} \quad 
\frac{\partial^2 \varphi_0}{\partial x_1^2} = O(x_1^{-3}),
\end{equation*}
as $x_1 \to \infty$, where the implicit constant is uniform in $K$.
\end{enumerate}
\end{enumerate}
\end{theorem}

\begin{example}\label{exm:1} When $k>1$, in general we can not extend
  $\varphi_{0}$ to a continuous function on $\ov{\bb{R}_{>\kappa}}^{k}
  \times K$ as the following example shows. Put $g=1$, $k=2$, $m=0$,
  $\tmatrix _{1}=\tmatrix _{2}=1$, $\tvector_{1}=1$, $\tvector_{2}=2$, $\tmatfun=0$, $\kappa=1$ and
  $\tvectfun=1$. Then
  \begin{displaymath}
    \varphi_{0}=\varphi-f=\frac{2(x_{1}+2x_{2})+1}{x_1+x_2}.
  \end{displaymath}
  The sequences $\{(n,n)\}_{n\ge 1}$ and $\{(n,2n)\}_{n\ge 1}$
converge, when $n\to \infty$, to the point $(\infty,\infty)\in
\ov{\bb{R}_{>1}}^{2}$. Nevertheless
\begin{displaymath}
  \lim_{n\to \infty} \varphi_{0}(n,n)=3,\qquad 
  \lim_{n\to \infty} \varphi_{0}(n,2n)=\frac{10}{3},
\end{displaymath}
showing that $\varphi_{0}$ can not be continuously extended to
$\ov{\bb{R}_{>1}}^{2}$. 
\end{example}

Before starting the proof of Theorem \ref{thm:main_technical} we recall a few easy statements related to
Schur complements and inverting a symmetric block matrix. For a
symmetric block matrix
\[ M = \begin{pmatrix}[c|c] A & B \\
\hline 
B^t & C \\
\end{pmatrix} \]
with $C$ invertible we call $A - BC^{-1}B^t$ the \emph{Schur
  complement} of the block $C$ in $M$. We have a product decomposition 
\[ M = \begin{pmatrix}[c|c] A & B \\
\hline 
B^t & C \\
\end{pmatrix} = \begin{pmatrix}[c|c] 1 & BC^{-1} \\
\hline 
0 & 1 \\
\end{pmatrix} 
\begin{pmatrix}[c|c] A - BC^{-1}B^t & 0 \\
\hline 
0 & C \\
\end{pmatrix} 
\begin{pmatrix}[c|c] 1 & 0 \\
\hline 
C^{-1}B^t & 1 \\
\end{pmatrix}
\, .  \]
In particular, $M$ is invertible if and only if $A-BC^{-1}B^t$ is invertible, and if these conditions are satisfied we have
\[ M^{-1} = \begin{pmatrix}[c|c] (A-BC^{-1}B^t)^{-1} & -(A-BC^{-1}B^t)^{-1}BC^{-1}  \\
\hline 
-C^{-1}B^t(A-BC^{-1}B^t)^{-1} & C^{-1} + C^{-1}B^t(A-BC^{-1}B^t)^{-1}BC^{-1} \\
\end{pmatrix} \, . \]
Also, if $M$ is positive semidefinite, then so is the Schur complement 
$A - BC^{-1}B^t$.

\subsection{Proof of the technical lemma}

First we observe that, if $k=0$, then $\varphi$ is a continuous function on a compact set, hence is bounded. Moreover, the
function $f$ is zero. Thus the statements are trivially true and we
are reduced to the case $k>0$ and hence $g>0$.

Assume that we have already shown that $\abs{\varphi-f}$ is bounded on
$\bb{R}_{>\kappa'}^k \times K$. Then, for each
$(x_1,\ldots,x_k,\lambda) \in \bb{R}_{>\kappa'}^k  \times K$ we have
\[ \lim_{\mu \to \infty} \frac{1}{\mu} \varphi(\mu x_1,\ldots,\mu
x_k,\lambda) = \lim_{\mu \to \infty} \frac{1}{\mu} f(\mu
x_1,\ldots,\mu x_k) \, . \] 
The latter limit exists and is equal to $f(x_1,\ldots,x_k)$ by
weight-one-homogeneity of $f$. Thus the recession function of
$\varphi$ exists and agrees with $f$. In consequence, in order to
prove Theorem \ref{thm:main_technical}.\ref{main_technical_bounddiff} and
\ref{thm:main_technical}.\ref{main_technical_fbound} we only need to 
show the boundedness of $\abs{\varphi-f}$ and of $f$ on the required
subsets. 

We next show that we can assume a simplifying hypothesis.

\begin{definition}\label{def:1} We say that the set of symmetric positive semidefinite matrices
  $\tmatrix _{1},\dots, \tmatrix _{k}$ satisfies the \emph{flag condition}
  if $\on{Ker}(\tmatrix_i) \subseteq \on{Ker}(\tmatrix_{i+1})$, for
  $i=1,\dots,k-1$.
\end{definition}

Consider the subset
\[ U = \{ 0 < x_1 \leq x_2 \leq \cdots \leq x_k \} \subset \bb{R}_{>0}^k \, .\]
Since
\[ \bb{R}_{>\kappa}^k = \bigcup_{\sigma \in \mf{S}_k} \left( \sigma^{-1}U \cap \bb{R}_{>\kappa}^k \right) \]
and
\[ \Delta^0 = \bigcup_{\sigma \in \mf{S}_k} \left( \sigma^{-1}U \cap \Delta^0 \right) \, , \]
by symmetry it is enough to prove the boundedness of $|\varphi-f|$ in
$U \cap \bb{R}_{>\kappa}^k$ and of $f$ in $U \cap \Delta^0$. Writing
$y_1=x_1$, $y_i=x_i-x_{i-1}$ for $i=2,\ldots,k$ we find that
$x_i = \sum_{j=1}^i y_j$ and that $U \cap \bb{R}_{>\kappa}^k$ is
parametrized by the set $y_1 > \kappa$, $y_2, \ldots, y_k \geq 0$.

Note that
\begin{displaymath}
  \sum_{i=1}^k x_i \tmatrix_i = \sum_{i=1}^k y_i \sum_{j=i}^k \tmatrix_j
\quad\text{and}\quad
\sum_{i=1}^k x_i \tmatrix_i\tvector_i = \sum_{i=1}^k y_i \sum_{j=i}^k
\tmatrix_j\tvector_j.
\end{displaymath}

\begin{lemma}
Writing $\tilde{\tmatrix}_i = \sum_{j=i}^k \tmatrix_j$ we have that
$\mathrm{Ker}\, \tilde{\tmatrix}_i \subseteq \mathrm{Ker}\,
\tilde{\tmatrix}_{i+1}$.
Moreover we have
$\Im(\tilde \tmatrix_{i})=\sum_{j=i}^{k} \Im( \tmatrix_j )$. 
\end{lemma}
\begin{proof} We first observe that, if $\tmatrix$ is a symmetric positive
  semidefinite real matrix, then $\tmatrix x=0$ if and only if
  $x^t\tmatrix x=0$. Indeed, clearly $\tmatrix x=0$ implies
  $x^{t}\tmatrix x=0$. Conversely, 
  assume that $x^{t}\tmatrix x=0$ and let $y$ be any vector. Then, for all
  $\lambda \in \rr$, 
  \begin{displaymath}
    0\le (y+\lambda x)^{t}\tmatrix (y+\lambda x)=y^{t}\tmatrix
    y+2\lambda y^{t}\tmatrix x  
  \end{displaymath}
  which implies that $y^{t}\tmatrix x=0$. Therefore $\tmatrix x=0$. 

  We show that this observation implies that $\Ker \tilde
  \tmatrix _{i}=\bigcap_{j=i}^{k}\Ker \tmatrix _{j}$. We have $x\in \Ker \tilde
  \tmatrix _{i}$ if and only if
  \begin{displaymath}
    0=x^{t}\tilde \tmatrix _{i}x=\sum_{j=i}^{k}x^{t}\tmatrix _{j}x.
  \end{displaymath}
  Since the matrices $\tmatrix _{j}$ are positive semidefinite this implies
  that $x^{t}\tmatrix _{j}x=0$, $j=i,\dots,k$. Therefore $x\in
  \bigcap_{j=i}^{k}\Ker \tmatrix _{j}$. The converse is clear. As a result
  \begin{displaymath}
    \mathrm{Ker}\, \tilde{\tmatrix}_i =\bigcap_{j=i}^{k}\Ker \tmatrix
    _{j} \subseteq \bigcap_{j=i+1}^{k}\Ker \tmatrix _{j} = \mathrm{Ker}\,
\tilde{\tmatrix}_{i+1}.  
  \end{displaymath}

  Since, for a symmetric positive semidefinite matrix  $\tmatrix $, the image
  $\Im (\tmatrix)$ is the orthogonal complement of $\Ker(\tmatrix )$ we deduce
  \begin{displaymath}
    \Im(\tilde{\tmatrix _{i}})=\Ker(\tilde{\tmatrix _{i}})^{\perp}=\Big(\bigcap_{j=i}^{k}\Ker( \tmatrix
    _{j})\Big)^{\perp}=\sum_{j=i}^{k}\Ker( \tmatrix_{j})^{\perp}=
    \sum_{j=i}^{k}\Im( \tmatrix_{j}) \, .
  \end{displaymath}
This proves the lemma.
\end{proof}
It follows from the Lemma that
there exist vectors $\tilde{\tvector}_i \in \bb{R}^g$ such that
\[ \sum_{j=i}^k \tmatrix_j\tvector_j = \tilde{\tmatrix}_i \tilde{\tvector}_i \, . \]
Replacing $\tmatrix_i$ by $\tilde{\tmatrix}_i$, $x_i$ by $y_i$ and $\tvector_i$
by $\tilde{\tvector}_i$ we are reduced to proving the boundedness of $\abs{\varphi-f}$ on
$\bb{R}_{>\kappa} \times \bb{R}_{\ge 0}^{k-1} \times K$ and of $f$ on
the set
\[ \{ (x_1,\ldots,x_k) \in \bb{R}_{\ge 0}^k : \, x_1 >0 \, , \, x_i \geq 0
\  \textrm{for all} \ i>1,\ \sum _{i=1}^{k}(k-i+1)x_{i}=1 \} \]  
under the extra hypothesis that the matrices $\tmatrix
_{1},\dots,\tmatrix_{k}$ satisfy the flag condition from Definition
\ref{def:1}. Clearly, by the homogeneity of $f$ it is enough to prove the
boundedness of $ f$ on the set
\[ H=\{ (x_1,\ldots,x_k) \in \bb{R}_{\ge 0}^k : \, x_1 >0 \, , \, x_i \geq 0
\  \textrm{for all} \ i>1,\ \sum _{i=1}^{k}x_{i}=1 \}. \]  

From now on we assume the flag condition and we write
$r_i = \mathrm{rk}(\tmatrix_i)$. Then $r=r_1 \ge \cdots \ge r_k\ge 1$.
Thanks to the flag condition, we can assume furthermore that the
basis of $\rr^{g}$ has been chosen 
in such a way that
\begin{equation}
\tmatrix'_i = \begin{pmatrix}[c|c] \tmatrix_i''& 0\\
\hline 
0 & 0\\
\end{pmatrix},
\end{equation}
with $\tmatrix_i''\in S^{++}_{r_{i}}(\rr)$.

 The following is our main estimate.
\begin{lemma} \label{main_estimate} There exists a constant $c$ such
  that for all $1 \leq \alpha, \beta \leq r$ and all $(x_1,\ldots,x_k)
  \in \bb{R}_{>0} \times \bb{R}_{\geq 0}^{k-1}$ we have the following
  bound on the $\alpha, \beta$-entry in the inverse of the $r$-by-$r$     matrix
  $\sum_{i=1}^k x_i \tmatrix'_i$:
  \[ \left|\Big( \sum_{i=1}^k x_i \tmatrix'_i \Big)^{-1}_{\alpha \beta}\right| \leq
  \frac{c}{\displaystyle \sum_{j \colon r_j \geq \min(\alpha,\beta)} x_j} \leq
  \frac{c}{x_1} . \] 
\end{lemma}
\begin{proof} This follows immediately from two intermediate results:
\begin{inclaim}
There exists a constant $c_{1}>0$ such that for all $(x_1,\ldots,x_k)
\in \bb{R}_{>0} \times \bb{R}_{\geq 0}^{k-1}$ we have the bound 
\begin{equation*}
\on{det}\Big(\sum_{i=1}^k x_i \tmatrix_i' \Big)\ge c_{1} 
\prod_{j=1}^r \sum_{i:r_i \ge j}x_i  > 0. 
\end{equation*}
\end{inclaim}

To prove this claim, define the $r$-by-$r$ matrix
\begin{equation}
J_i = \begin{pmatrix}[c|c] \on{Id}_{r_i}& 0\\
\hline 
0 & 0\\
\end{pmatrix}. 
\end{equation}
Since $\tmatrix''_{i}$ is positive definite, there exists $\epsilon >0$
such that for all $i$, we have that $\tmatrix_i' - \epsilon J_i$ is
positive semidefinite. Then 
\begin{equation}\label{eq:18}
\sum_i x_i\tmatrix_i'  = \sum_i x_i(\tmatrix_i' - \epsilon J_i) + \sum_i x_i \epsilon J_i.
\end{equation}
The proof of the following lemma is left to the reader.

\begin{lemma} \label{lemm:2}
Let $A$, $B$ be positive semidefinite symmetric $r \times r$
matrices. Then $\det(A + B) \ge \det(A)+ \det(B)$. \hfill $\square$
\end{lemma}

Combining Lemma \ref{lemm:2} with equation \eqref{eq:18}, we deduce
\begin{equation*}
\on{det} \Big(\sum_i x_i\tmatrix_i'\Big)  \ge \on{det}  \Big(\sum_i x_i
\epsilon J_i\Big)  = \epsilon^r \prod_{j=1}^r \sum_{i:r_i \ge j}x_i  >
0 
\end{equation*}
as required. 
The second intermediate result is as follows:
\begin{inclaim}
There exists a constant $c_{2} > 0$ such that for all $1 \le \alpha,
\beta \le r$ and all $(x_1,\ldots,x_k) \in \bb{R}_{>0} \times
\bb{R}_{\geq 0}^{k-1}$ we have the following bound on the cofactors of
the matrix $\sum_{i=1}^k x_i \tmatrix'_i$: 
\begin{equation*}
\left|\on{cof}_{\alpha, \beta}\Big(\sum_{i=1}^k x_i\tmatrix'_i\Big)
\right|\le c_{2} \prod_{\stackrel{\alpha '=1}{\alpha ' \neq
    \min(\alpha, \beta)}}^r \sum_{i : r_i \ge \alpha ' } x_i.
\end{equation*}
\end{inclaim}
To prove this second claim, write $A=\sum_i x_i \tmatrix_i'$. Then there
is a constant $c_3$ such that for $1
\leq \alpha' ,\beta'  \leq r$ one has 
\[ \big| A_{\alpha' ,\beta' } \big| \leq c_{3} \sum_{i : r_i \geq \max(\alpha' ,\beta' )} x_i \leq
c_{3} \sum_{i : r_i \geq \alpha' } x_i . \]
Let $\sigma \colon \{ 1,\ldots,\hat{\alpha},\ldots,r\} \isom 
\{ 1,\ldots,\hat{\beta},\ldots,r\}$ be a bijection (the $\hat{}$ means ``omit''). Then
\[ \prod_{\alpha ' \neq \alpha} \big| A_{\alpha ',\sigma(\alpha
  ')}\big | \leq
c_{3}^{r-1} \prod_{\alpha ' \neq \alpha} \sum_{i : r_i \geq \alpha '} x_i \]
and since $A_{\alpha ',\sigma(\alpha ')}=A_{\sigma(\alpha '),\alpha '}$ we also have
\[ \prod_{\alpha ' \neq \alpha} \big| A_{\alpha ',\sigma(\alpha ')}\big| \leq
c_{3}^{r-1} \prod_{\alpha ' \neq \beta} \sum_{i : r_i \geq \alpha '} x_i \, . \]
Choosing the smaller upper bound of the two we find
\[ \prod_{\alpha ' \neq \alpha}\big| A_{\alpha ',\sigma(\alpha ')} \big| \leq
c_{3}^{r-1} \prod_{\alpha ' \neq \min(\alpha,\beta)} \sum_{i : r_i \geq \alpha '} x_i \]
and hence
\[ \big|\on{cof}_{\alpha, \beta}(A)\big| \leq (r-1)!c_{3}^{r-1} 
\prod_{\alpha ' \neq \min(\alpha,\beta)} \sum_{i : r_i \geq \alpha '} x_i  \, . \]
This proves the second claim and, consequently, Lemma
\ref{main_estimate}.
\end{proof}
\begin{proof}[Proof of Theorem
  \ref{thm:main_technical}~\eqref{main_technical_fbound}] 
From Lemma \ref{main_estimate} we deduce the existence of a constant
$c_{4}>0$ such that, for all $1\le \alpha ,\beta \le r$, 
\begin{equation*} \begin{aligned}
 \left|\left( \sum x_i \tmatrix'_i \tvector'_i \right)_\alpha^t \left( \sum x_i
     \tmatrix'_i\right)^{-1}_{\alpha,\beta} \left( \sum x_i \tmatrix'_i
     \tvector'_i \right)_{\beta} \right| 
 & \leq c_{4} \cdot \frac{\left(\sum_{j \colon r_j \geq \alpha} x_j\right) \left(  \sum_{i
     \colon r_i \geq \beta} x_i \right) }{\sum_{j \colon r_j \geq
     \min(\alpha,\beta)} x_j } \\ 
 & = c_{4} \cdot \sum_{j \colon r_j \geq \max(\alpha,\beta)} x_j 
\end{aligned} 
\end{equation*}
and hence
\[   0 \leq f   = \sum_{\alpha,\beta} \left( \sum x_i \tmatrix'_i \tvector'_i \right)_\alpha^t \left( \sum x_i \tmatrix'_i\right)^{-1}_{\alpha,\beta} \left( \sum x_i \tmatrix'_i \tvector'_i \right)_{\beta} \leq c_{4} \sum_{\alpha,\beta}
\sum_{j \colon r_j \geq \max(\alpha,\beta)} x_j \, . 
\]
This is clearly bounded on $H$. This proves Theorem
\ref{thm:main_technical}~\eqref{main_technical_fbound}. 
\end{proof}
\begin{proof}[Proof of Theorem \ref{thm:main_technical}~\eqref{main_technical_bounddiff}] We start by noting that
\[ P = \begin{pmatrix}[c|c] \sum x_i \tmatrix_i' + \tmatfun_{11} & \tmatfun_{12}\\
\hline 
\tmatfun_{21} & \tmatfun_{22} \\
\end{pmatrix} \, , \]
with $\tmatfun_{22}$ invertible. Moreover, as $P$ is invertible, so is
the Schur complement $\sum x_i \tmatrix'_i + \tmatfun_{11} -
\tmatfun_{12}\tmatfun_{22}^{-1}\tmatfun_{21}  $ of $\tmatfun_{22}$ in $P$. If we put
\[ Q = \left( \sum x_i \tmatrix'_i + \tmatfun_{11} - \tmatfun_{12}\tmatfun_{22}^{-1}\tmatfun_{21}  \right)^{-1} \]
then we get
\begin{equation} \label{oominverse}  P^{-1} = \begin{pmatrix}[c|c] Q& -Q\tmatfun_{12}\tmatfun_{22}^{-1} \\
\hline 
-\tmatfun_{22}^{-1}\tmatfun_{21}Q & \tmatfun_{22}^{-1} + \tmatfun_{22}^{-1}\tmatfun_{21}Q\tmatfun_{12}\tmatfun_{22}^{-1} \\
\end{pmatrix} \, . 
\end{equation}
Write $A = \sum x_i \tmatrix'_i$ and $M= \tmatfun_{11} - \tmatfun_{12}\tmatfun_{22}^{-1}\tmatfun_{21}$ so that $Q = (A+M)^{-1}$. Recall that $A$ is invertible, so that $Q=(\mathrm{Id}_r+A^{-1}M)^{-1}A^{-1}$.
\begin{claim}\label{claim:seriesQ} There exists a $\kappa' > \kappa$
  such that the series 
\[ A^{-1} - A^{-1}MA^{-1} + A^{-1}MA^{-1}MA^{-1} + \cdots + (-1)^m (A^{-1}M)^mA^{-1} + \cdots \]
converges to $Q$ on $\bb{R}_{>\kappa'} \times \bb{R}_{\geq 0}^{k-1} \times K$.
\end{claim}
\begin{proof} The entries of the matrix $M$ are continuous functions
  on the compact set $K$, hence bounded. Let $c$ be the constant of Lemma
  \ref{main_estimate}, choose \[\kappa' > \max( c r^2 \max(M_{\alpha
    \beta}),c,\kappa ) \]  
and put $\epsilon = c r^2\max(M_{\alpha\beta})/\kappa'$. Note that $0
< \epsilon < 1$. Moreover, by Lemma \ref{main_estimate} and the
condition $x_1\geq \kappa'$,
\[ 
\left( \abs{(A^{-1}M)^mA^{-1}} \right)_{\alpha \beta}  
\leq \frac{c}{x_1} \frac{(c\max(M_{\alpha \beta})r^{2})^{m}}{\kappa'^m} < \epsilon^m . \]
It follows that the series converges absolutely. By construction, the limit of the series is $(A+M)^{-1}=Q$ finishing the proof of the claim.
\end{proof}

Write $M_{1}=(\mathrm{Id}_r+MA^{-1})^{-1}$ and
$M_{2}=(\mathrm{Id}_r+A^{-1}M)^{-1}$. Then $Q = A^{-1}M_1 = M_2
A^{-1}$. An argument similar to that of Claim \ref{claim:seriesQ}
shows that the entries of  $M_1$ and $M_2$ are bounded on the set
$\bb{R}_{> \kappa'} \times \bb{R}_{\ge 0}^{k-1} \times K$.  

We deduce from Lemma \ref{main_estimate} that there is a constant
$c_{2}$ such that 
\[\abs{ Q_{\alpha \beta}} \le \frac{c_{2}}{\sum_{j \colon r_j \geq
    \min(\alpha,\beta)}x_j} \] 
on the same set. It follows that
\[ \left| \Big( Q (\sum x_i \tmatrix'_i \tvector'_i) \Big)_\beta\right| \quad \textrm{and}\quad
\left| \Big( (\sum x_i \tmatrix'_i \tvector'_i)^t Q \Big)_\alpha\right| \]
are bounded on $\bb{R}_{>\kappa'} \times \bb{R}_{\ge 0}^{k-1} \times
K$. Moreover, since $Q-A^{-1} = A^{-1}M_3 A^{-1}$ with again $M_3$ having
bounded entries, we deduce that there is another constant $c_{3}$ such
that 
\[ \left|\left( Q- A^{-1} \right)_{\alpha \beta} \right|\le \frac{c_{3}}{ \left( \sum_{j \colon r_j \ge \alpha} x_j \right)\left( \sum_{i \colon r_i \ge \beta} x_i\right) }, \]
and consequently
\[ \left| \left( \sum x_i \tmatrix'_i \tvector'_i \right)^t \left( Q-A^{-1}
  \right) \left( \sum x_i \tmatrix'_i \tvector'_i \right) \right| \]
is bounded. Finally, to prove that $|\varphi-f|$ is bounded we compute
\begin{multline*}
  \varphi-f=  \left( \sum x_i \tmatrix'_i \tvector'_i \right)^t \left(
               Q-A^{-1} \right) \left( \sum x_i \tmatrix'_i \tvector'_i
               \right) 
  + 2 \tvectfun_{1}^{t}Q\left( \sum x_i \tmatrix'_i \tvector'_i
               \right)\\ + \tvectfun_{1}^{t}Q\tvectfun_{1}
   - 2 \tvectfun_{2}^{t}\tmatfun_{22}^{-1}\tmatfun_{21}Q\left( \sum x_i \tmatrix'_i \tvector'_i
               \right)\\
   - 2 \tvectfun_{2}^{t}\tmatfun_{22}^{-1}\tmatfun_{21}Q\tvectfun_1
  + \tvectfun_{2}^{t}(\tmatfun_{22}^{-1} + \tmatfun_{22}^{-1}\tmatfun_{21}Q\tmatfun_{12}\tmatfun_{22}^{-1})\tvectfun_{2}
\end{multline*}
 and we use the previously obtained bounds. This proves Theorem 
\ref{thm:main_technical}~\eqref{main_technical_bounddiff}.
\end{proof}
\begin{proof}[Proof of Theorem \ref{thm:main_technical}~\eqref{main_technical_k=1}]
From now on we assume that $k=1$ so we have $\varphi \colon \bb{R}_{>\kappa}
\times K \to \bb{R}$ and $f \colon \bb{R}_{>0} \to
\bb{R}$. Explicitly,  
\[ \varphi(x_1,\lambda) = (\tmatrix_1x_1\tvector_1 + \tvectfun)^t P^{-1}(\tmatrix_1x_1\tvector_1+\tvectfun) \, \]
with $P = \tmatrix_1x_1 + \tmatfun$, and $f=\tvector_1^t \tmatrix_1 \tvector_1x_1$. Recall
that we write $\varphi_0 = \varphi-f$. Put $w_0 = \tvectfun-\tmatfun\tvector_1$.
\begin{lemma} \label{simplifyg}
We have
\begin{equation*}
\varphi_0(x_1, \lambda) = 2\tvectfun^t \tvector_1 - \tvector_1^t\tmatfun\tvector_1  + w_0^tP^{-1}w_0. 
\end{equation*}
\end{lemma}
\begin{proof} We compute
\begin{equation*} \begin{split}
\varphi_0(x_1, \lambda) & = (\tmatrix_1x_1\tvector_1 + \tvectfun)^t
P^{-1}(\tmatrix_1x_1\tvector_1+\tvectfun) - 
\tvector_1^t \tmatrix_1 \tvector_1x_1 \\
& = (w_{0}+P \tvector_{1})^{t}  P^{-1} (w_{0}+P
\tvector_{1})-\tvector_1^t \tmatrix_1 \tvector_1x_1\\
&= w_{0}^{t}  P^{-1} w_{0}+ 2 \tvector_{1}^{t}w_0+ \tvector_1^t P \tvector_1-\tvector_1^t
\tmatrix_1 \tvector_1x_1\\
&= w_{0}^{t}  P^{-1} w_{0}+2 \tvector_{1}^{t}\tvectfun-\tvector_1^t \tmatfun \tvector_1.
\end{split}
\end{equation*}
\end{proof}

We continue to assume that $k=1$. It follows that $\tmatrix_1'$ is invertible.
\begin{lemma} \label{expandoominverse}
In the above notation and with $k=1$, we have
\begin{equation*}
P^{-1} = \begin{pmatrix}[c|c] 0&0\\
\hline 
0 & \tmatfun_{22}^{-1}\\
\end{pmatrix}  +  \frac{1}{x_1}
\begin{pmatrix}[c|c] \tmatrix'^{-1}_1 & -\tmatrix'^{-1}_1 \tmatfun_{12}\tmatfun_{22}^{-1} \\
\hline 
-\tmatfun_{22}^{-1}\tmatfun_{21} \tmatrix'^{-1}_1 & \tmatfun_{22}^{-1} \tmatfun_{21} \tmatrix'^{-1}_1 \tmatfun_{12} \tmatfun_{22}^{-1} \\
\end{pmatrix}
+ O(x_1^{-2}) 
\end{equation*}
and
\begin{equation*}
P^{-1}\tmatrix_1 = \frac{1}{x_1} \begin{pmatrix}[c|c] 1&0\\
\hline 
-\tmatfun_{22}^{-1}\tmatfun_{21} & 0 \\
\end{pmatrix} + O(x_1^{-2}) 
\end{equation*}
as $x_1 \to \infty$, where the implicit constants are uniform in $K$. 
\end{lemma}
\begin{proof} From equation (\ref{oominverse}) we obtain
\begin{equation*} P^{-1} = 
\begin{pmatrix}[c|c] 0&0\\
\hline 
0 & \tmatfun_{22}^{-1}\\
\end{pmatrix} +  \begin{pmatrix}[c|c] Q & -Q\tmatfun_{12}\tmatfun_{22}^{-1} \\
\hline 
-\tmatfun_{22}^{-1}\tmatfun_{21} Q & \tmatfun_{22}^{-1} \tmatfun_{21} Q \tmatfun_{12} \tmatfun_{22}^{-1} \\
\end{pmatrix} \,  .
\end{equation*}
Also recall that $Q = (\mathrm{Id}_r + A^{-1}M)^{-1} A^{-1}$ with $A =
\tmatrix_1'x_1$ and $M$ bounded. This yields $Q =x_1^{-1}
\tmatrix'^{-1}_1+ O(x_1^{-2})$ as $x_1 \to \infty$. The first estimate
readily follows. Upon recalling that
\[ \tmatrix_1 = \begin{pmatrix}[c|c] \tmatrix_1'&0\\
\hline 
0 & 0 \\
\end{pmatrix} \]
the second estimate also follows. 
\end{proof}
To finish the proof of Theorem \ref{thm:main_technical}.\ref{main_technical_k=1}, note that by combining 
Lemma \ref{simplifyg} and Lemma \ref{expandoominverse} that
\[ \varphi_0(x_1,\lambda) = 2\tvectfun^t \tvector_1-\tvector_1^t\tmatfun\tvector_1 + w_0^t \begin{pmatrix}[c|c] 0&0\\
\hline 
0 & \tmatfun_{22}^{-1}\\
\end{pmatrix} w_0 + O(x_1^{-1})  \]
as $x_1 \to \infty$.
From this it is immediate that $\varphi_0$ extends continuously to a
function from $\ov{\bb{R}_{>\kappa}} \times K$ to $\bb{R}$.
Next, from Lemma \ref{simplifyg} we have
\[  \frac{ \partial \varphi_0}{\partial x_1} = -w_0^t P^{-1}
\tmatrix_1 P^{-1} w_0 , \quad \frac{ \partial^2
  \varphi_0}{\partial x_1^2} =
2w_0^t P^{-1} \tmatrix_1 P^{-1} \tmatrix_1 P^{-1} w_0 . \]
Combining this with Lemma \ref{expandoominverse} we find the estimates
\[ \frac{ \partial \varphi_0}{\partial x_1} = O(x_1^{-2}) , \quad
\frac{ \partial^2 \varphi_0}{\partial x_1^2} = O(x_1^{-3}) ,     \]
completing the proof of Theorem \ref{thm:main_technical}.\ref{main_technical_k=1}.
\end{proof}

\subsection{On the recession function of a normlike function}
\label{propertiesnormlike}

Let $f \colon \bb{R}^k_{>0} \to \bb{R}$ be the recession function
of the normlike function $\varphi$ associated to 
$((\tmatrix_i), (\tvector_i), \tvectfun, \tmatfun)$ as above. The purpose of this section is
to list a number of useful properties of $f$. 

\begin{proposition}\label{prop:convex}
The function $f$ is convex, that is, for all $x, y \in \bb{R}_{>0}^k$
and all $\lambda \in [0,1]$
we have $f(\lambda x + (1-\lambda)y) \leq \lambda f(x) + (1-\lambda)f(y)$.   
\end{proposition}
\begin{proof}
Example~3.4 on p.~90 of \cite{boyd} states that the function $h_g
\colon \bb{R}^g \times S^{++}_{g}(\bb{R})\ra 
\rr$ given by $
h_g(x,Y) = x^t Y^{-1} x$ is
convex. The function $f \colon \bb{R}_{>0}^k \to \bb{R}$ is the
composition of $h_g$ with the linear map 
\[ \bb{R}_{>0}^k \to \bb{R}^g \times S_g^{++}(\bb{R}) , \quad
(x_1,\ldots,x_k) \mapsto \left( \sum_{i=1}^k x_i \tmatrix_i'\tvector_i'\ ,\,
  \sum_{i=1}^k x_i\tmatrix'_i \right). \]
Since the composition of a linear map followed by a convex function is
again convex, we deduce that $f$ is convex.
\end{proof}
\begin{proposition} \label{extendconvex}
The function $f$ extends to a continuous function $\ov{f} \colon
\bb{R}^k_{\ge 0} \to \bb{R}_{\ge 0}$. The function $\ov{f}$ is
homogeneous of weight one and convex.
\end{proposition}
\begin{proof} By Theorem
  \ref{thm:main_technical}~\eqref{main_technical_fbound} we know that
  the function $f$ is bounded 
  on the open standard simplex $\Delta^0$. Define
\begin{equation*}
\ov{\ov{f}} \colon \Delta \ra \bb{R}_{\ge 0}
\end{equation*}
by the formula
\begin{equation*}
\ov{\ov{f}}(x_1, \ldots, x_k)= \on{inf}_{(p_l)_l \ra (x_1, \ldots, x_k)} \on{liminf}_{l \ra\infty} f(p_l);
\end{equation*}
here the infimum is over sequences in $\Delta^0$ tending to the
point $(x_1, \ldots, x_k)$. This function $\ov{\ov{f}}$ is
well-defined because $f$ is bounded on $\Delta^0$. It follows easily
from the definition of $\ov{\ov{f}}$ that $\ov{\ov{f}}$ is convex
and lower semi-continuous. Since $\Delta$ is a convex polytope, it
follows from \cite[Theorem~10.2]{Roc70} that $\ov{\ov{f}}$ is
continuous. Now 
extend $\ov{\ov{f}}$ to $\bb{R}_{\ge 0}^k \setminus \{0\}$ by
homogeneity. By sending in addition $0$ to $0$ we obtain the required
continuous and convex function $\ov{f} \colon \bb{R}^k_{\ge 0} \to
\bb{R}_{\ge 0}$.   
\end{proof}
We can make the function $\ov{f}$ explicit as follows. Let $I \sub
\{1, \ldots, k\}$ be 
any subset, and set $J = \{1, \ldots, k \} \setminus I$. We
consider the restriction of $\ov{f}$ to the subset $\bb{R}_{> 0}^I
\sub \bb{R}_{\ge 0}^k$ given by setting $x_j$ equal to zero for all $j
\in J$. Let  
\begin{equation*}
r_I = \on{rk}\left( \sum_{i \in I} x_i \tmatrix_i\right) : x_i > 0, 
\end{equation*}
and for $i \in I$ set 
\begin{equation}
\tmatrix_i = \begin{pmatrix}[c|c] \tmatrix_i''& 0\\
\hline 
0 & 0\\
\end{pmatrix}
\end{equation}
where $\tmatrix_i''$ has size $r_I$, and similarly  
\begin{equation*}
\tvector_i = \mat{\tvector_i''\\
\hline
\star}, 
\end{equation*}
where $\tvector_i''$ has length $r_I$.

Note that, if $I  \not =  \emptyset$, then  $r_I \ge 1$. Let $K
\subset \bb{R}_{>0}^J$ be an arbitrary compact subset. Write $x_I =
(x_i)_{i \in I}$ and $x_J = (x_j)_{j \in J}$.
We define the function
\[ f_I \colon \bb{R}_{>0}^I \times K \to \bb{R} \, , \quad
(x_I;x_J) \mapsto f(x_1,\ldots,x_k) \, . \]
Write
\[ \tvectfun(x_J) = \sum_{j \in J} x_j\tmatrix_j' \tvector'_j \, , \quad \tmatfun(x_J) =
\sum_{j \in J} \tmatrix'_j x_j \, , \]
then we see that
\[ f_I(x_I;x_J) = \left( \sum_{i \in I} x_i \tmatrix'_i\tvector'_i + \tvectfun(x_J) \right)^t 
\left( \sum_{i \in I} x_i \tmatrix_i' + \tmatfun(x_J) \right)^{-1} \left(
  \sum_{i \in I} x_i \tmatrix'_i\tvector'_i + \tvectfun(x_J) \right) \] 
and hence by Theorem \ref{thm:main_technical} $f_I$ has a recession function
$\on{rec}( f_I) \colon \bb{R}^I_{>0} \to \bb{R}$ which can be written 
\[ 
\on{rec}( f_I(x_I)) = \left(\sum_{i\in I}x_i\tmatrix''_i\tvector''_i  \right)^t
\left(\sum_{i\in I}x_i\tmatrix''_i \right)^{-1} \left(\sum_{i\in
    I}x_i\tmatrix''_i\tvector''_i  \right),
\]
when $I\not = \emptyset$, and $\on{rec} (f_\emptyset)=0$. 
Note that $\on{rec}( f_I)$ is independent of the choice of $K$. Also
note that, by Theorem \ref{thm:main_technical}, $\abs{f_I-\on{rec}( f_I) }$ is
bounded on $\bb{R}_{>0}^I \times K$. 
\begin{proposition} \label{restrict_to_faces}
Let $I \sub \{1, \ldots, k\}$ be any subset. 
We have
\begin{equation*}
\ov{f}|_{\bb{R}_{>0}^I} = \on{rec}( f_I).
\end{equation*}
\end{proposition}
\begin{proof} When $I=\emptyset$ the equality is trivially true. We
  assume that $I\not=\emptyset$. Choose $c \in \bb{R}^J_{>0}$ and
  $x_I \in \bb{R}^I_{>0}$ arbitrarily. By Theorem
  \ref{thm:main_technical}(\ref{main_technical_bounddiff}) with $K=\{c\}$, there
  exists a constant $\delta >0$ depending on $c$ and $x_I$ such that
  for all $\lambda \in \bb{R}_{>0}$ we have
\[ \abs{(\on{rec} (f_I))(\lambda x_I) - f_I(\lambda x_I;c) } \leq \delta \, . \]
We deduce that for all $\lambda \in \bb{R}_{>0}$ we have
\begin{equation}\label{eq:1}
  \abs{(\on{rec} (f_I))(x_I) -  f  (x_I,\frac{c}{\lambda}) } \leq
  \frac{\delta}{\lambda} . 
\end{equation}

As $\ov{f}$ extends $f$ continuously we have
\[ \lim_{\lambda \to \infty} f(x_I,\frac{c}{\lambda}) = \ov{f}|_{\bb{R}_{>0}^I}(x_I)\]
independently of the choice of $c$. Combining with the bound
\eqref{eq:1}, we find that 
\[ (\on{rec}(f_I))(x_I) = \ov{f}|_{\bb{R}_{>0}^I}(x_I) , \]
as required.
\end{proof}
A special case of interest is when $|I|=1$. For each $1 \le i \le k$, set 
\begin{equation}
\tmatrix_i = \begin{pmatrix}[c|c] \tmatrix_i^{\on{e}}& 0\\
\hline 
0 & 0\\
\end{pmatrix}
\end{equation}
where $\tmatrix_i^{\on{e}}$ has size $r_i = \on{rk} \tmatrix_i$ and hence is positive definite; here $\on{e}$ is short for ``essential''. Similarly set 
\begin{equation*}
\tvector_i = \mat{\tvector_i^{\on{e}}\\
\hline
\star}, 
\end{equation*}
where $\tvector_i^{\on{e}}$ has length $r_i$. Define
\begin{equation*}
\mu_i = \tvector_i^t \tmatrix_i \tvector_i = (\tvector_i^{\on{e}})^t\tmatrix_i^{\on{e}}\tvector_i^{\on{e}} \, . 
\end{equation*}
Then $\mu_i \geq 0$ and we have for all $x_i >0$
\begin{equation*}
\ov{f}(0,\ldots, 0, x_i, 0, \ldots, 0) = (\on{rec} (f_{\{i\}})) (x_i)  
= x_i (\tvector_i^{\on{e}})^t (\tmatrix_i^{\on{e}})^t x_i^{-1}
(\tmatrix_i^{\on{e}})^{-1} x_i\tmatrix_i^{\on{e}} \tvector_i^{\on{e}} = x_i \mu_i
. 
\end{equation*}
In particular, the function $\ov{f}(0,\ldots, 0, x_i, 0, \ldots, 0) $
is homogeneous linear in $x_i$, and 
\[ \mu_i = \ov{f}( 0,\ldots, 0, 1, 0, \ldots, 0  ) \, . \]
We call $\mu_1,\ldots,\mu_k \geq 0$ the \emph{coefficients} associated
to $\varphi$.  

\section{Proofs of the main results}
\label{sec:proof-main-results}

In this section we prove our main results. We will continue to work with the ``diagonal
case'' where we consider the pullback Poincar\'e bundle $\ca{P}_\nu$
associated to a single admissible normal function $\nu$ of our family $\pi \colon Y \to
X$. As was explained at the beginning of Section
\ref{sec:norm-section}, by the biextension property of the Poincar\'e
bundle this is sufficient for the purpose of proving the main results
as stated in the introduction.

\subsection{Singularities of the biextension metric} \label{singularity}

In this section we will prove Theorem \ref{singbiext}. 
\begin{proof}[Proof of Theorem \ref{singbiext}]
Following Theorem \ref{asympt}, take 
\begin{itemize}
\item[-] a small enough $\epsilon >0$,
\item[-] matrices 
\begin{equation*}
\tmatrix_1, \ldots, \tmatrix_k \in S_g(\bb{R}) 
\end{equation*}
of positive rank such that $\Delta^{-1}\tmatrix_i \in M_g(\bb{Z})$ for $i=1,\ldots,k$,
\item[-] vectors 
\begin{equation*}
\tvector_1, \ldots, \tvector_k \in \bb{Q}^g
\end{equation*}
such that $\Delta^{-1}\tmatrix_i \tvector_i \in \bb{Z}^g$ for $i=1,\ldots,k$,
\item[-] bounded holomorphic maps $\alpha\colon  \Delta_\epsilon^n \ra \bb{C}^g$ and $\psi\colon  \Delta_\epsilon^n \ra \bb{P}^g$,
\end{itemize}
such that the multi-valued period mapping 
\begin{equation}
(\Omega,\delta) \colon U_\epsilon \cap X  \ra \mathcal{M} = \bb{H}_g \times \bb{C}^g
\end{equation}
of the variation of mixed Hodge structures $\hh(\nu)$ on $U_\epsilon$ is given by the formula
\begin{equation}
(\underline{q}) = (q_1, \ldots, q_n ) \mapsto \left(
  \sum_{j=1}^k\tmatrix_j\frac{\log q_j}{2\pi i} + \psi(\underline{q}),
  \sum_{j=1}^k\tmatrix_j\tvector_j\frac{\log q_j}{2\pi i} +
  \alpha(\underline{q}) \right)
\end{equation}
(recall that $U_\epsilon$ was defined in Section \ref{sec:statement_of_main}). 
Put $\tvectfun = 2\pi \Im \alpha$, $\tmatfun= 2\pi \Im \psi$, and define $\kappa \in
\bb{R}$ via $\kappa = -\log \epsilon$.
As above define the function $\varphi \colon \bb{R}_{>\kappa}^k \times
\Delta_{\epsilon}^n \to \bb{R}_{\ge 0}$ via 
\[ 
\varphi (x_1,\ldots,x_k; \underline{q} ) =
\left(\sum_{i=1}^kx_i\tmatrix_i\tvector_i + \tvectfun \right)^t
\left(\sum_{i=1}^kx_i\tmatrix_i + \tmatfun \right)^{-1}
\left(\sum_{i=1}^kx_i\tmatrix_i\tvector_i + \tvectfun \right). 
\] 

Choose any $0<\epsilon '<\epsilon $. The restriction of $\varphi$ to
$\bb{R}_{>\kappa}^k \times \overline{\Delta}_{\epsilon'}^n$ is then a
normlike function of dimension $k$. 
Let $f \colon \bb{R}^k_{>0} \to \bb{R}_{\ge 0}$ be the associated
recession function $f = \on{rec}( \varphi)$. Recalling the explicit expression
(\ref{recessionexplicit}) for $f$, the conditions 
\begin{equation*}
\Delta^{-1}\tmatrix_i \in S_g(\bb{R}) \cap M_g(\bb{Z}) \, , \quad
\Delta^{-1}\tmatrix_i \tvector_i \in \bb{Z}^g,
\end{equation*}
for each $i=1,\ldots,k$ imply that the $A_i$ and $A_ic_i$ are themselves integral and that $f$ is the quotient of two
homogeneous polynomials in $\bb{Z}[x_1,\ldots,x_k]$. In particular
$f \in \bb{Q}(x_1,\ldots, x_k)$. It is clear that $f$ is homogeneous
of weight one, and by Proposition \ref{prop:convex} the function $f$ is convex when seen as a real-valued function on $\bb{R}_{>0}^k$.

Let $s$ be an admissible  section of $\ca{P}_\nu$ over
$U_\epsilon \cap X$. Following Corollary \ref{explicitnorm} we have
\begin{equation*}
-\log\aabs{s} = -\log\abs{h} +
\varphi(-\log|q_1|,\ldots,-\log|q_k|;\underline{q})
\end{equation*}
on $U_\epsilon \cap X$ with $h$ a meromorphic function on
$U_\epsilon$, holomorphic on $U_\epsilon \cap X$. As $s$ is locally
generating over $U_\epsilon \cap X$ we have that $h$ has no zeroes or
poles on $U_\epsilon \cap X$. Hence there is a linear form
$l \in \bb{Z}[x_1,\ldots,x_k]$ and a holomorphic map
$u \colon U_\epsilon \to \bb{C}^*$ such that
\[ 
-\log|h| = l(-\log|q_1|,\ldots,-\log|q_k|) + \log|u| 
\]
on $U_\epsilon \cap X$. The image of $\overline
U_{\epsilon '}$ under the map $u$ is compact.

Put $f_s=f+l$ in
$\bb{Q}(x_1,\ldots,x_k)$. Then $f_s$ is again homogeneous of
weight one and convex as a function on $\bb{R}_{>0}^k$. Our claim is
that $f_s$ satisfies all the requirements of Theorem
\ref{singbiext}. We need to show first of all that $-\log \aabs{s} - f_s(-\log|q_1|,\ldots,-\log|q_k|)$ is bounded
on $\overline{U}_{\epsilon'} \cap X$ and extends continuously over
$\overline{U}_{\epsilon'} \setminus D^{\mathrm{sing}}$.

In order to see this, put $\varphi_0=\varphi-f$ on $\bb{R}_{>\kappa}^k \times
\Delta_\epsilon^n$. Then
\[ -\log \aabs{s}(\underline{q}) =
f_s(-\log|q_1|,\ldots,-\log|q_k|) + \log|u| +
\varphi_0(-\log|q_1|,\ldots,-\log|q_k|;\underline{q}) \]
on $U_\epsilon \cap X$. Note that $\log|u|$ extends in a continuous
and bounded manner over the whole of $\overline{U}_{\epsilon'}$. We
are reduced to 
showing that $\varphi_0(-\log|q_1|,\ldots,-\log|q_k|;\underline{q})$ is
bounded on $\overline{U}_{\epsilon'} \cap X$ and extends continuously over
$\overline{U}_{\epsilon'} \setminus D^{\mathrm{sing}}$.

For this we invoke Theorem \ref{thm:main_technical}.\ref{main_technical_bounddiff}. This readily gives the boundedness of
$\varphi_0$ via the map
\[ 
(-\log|\cdot|,\mathrm{id}) \colon (\Delta_\epsilon^*)^k \times
\Delta_\epsilon^n \to \bb{R}_{>\kappa}^k \times \Delta_\epsilon^n . 
\]
Let
$p \in (D \setminus D^\mathrm{sing})\cap \overline{U}_{\epsilon'}$. Up
to a change in the order of the variables, we can assume that the
coordinates of $p$ satisfy $q_1=0$, $q_i \neq 0$ for $i=2,\ldots,
N$.
We take a small polydisk $V_{\epsilon''} \subset
\overline{U}_{\epsilon'}$ of small 
radius $\epsilon''$ with center at $p$ such that $V_{\epsilon''} \cap X$
can be identified with
$\Delta_{\epsilon''}^* \times \Delta_{\epsilon''}^{n-1}$ and hence
$V_{\epsilon''} \cap D$ can be identified with the divisor $q_1=0$ on
$\Delta_{\epsilon''}^n$. Write
\[ \underline{r} = (r_2,\ldots,r_k) =
(-\log|q_2|,\ldots,-\log|q_k|) \]
for $\underline{q} \in V_{\epsilon''}$; then $\underline{r}$ can be
assumed to move through a compact subset $K' \subset
\bb{R}^{k-1}$.
Put $K'' = K' \times \overline{\Delta}_{\epsilon'}^n$. We define functions
$\varphi' \colon \bb{R}_{>\kappa} \times K'' \to \bb{R}_{\geq 0}$ and
$f' \colon \bb{R}_{>\kappa} \times K'' \to \bb{R}_{\ge 0}$ via
\[ \varphi'(x_1;\underline{r},\underline{q}) = \varphi(x_1,\underline{r};\underline{q}) \, , \quad f'(x_1;\underline{r}) = f(x_1,\underline{r}) \, . \]
Then both $\varphi'$, $f'$ are normlike of dimension one. Write
\[ \tmatrix_1 = \begin{pmatrix}[c|c] \tmatrix_1'& 0\\
\hline 
0 & 0\\
\end{pmatrix} \, , \quad \tmatrix_1' = \begin{pmatrix}[c|c] \tmatrix_1''& 0\\
\hline 
0 & 0\\
\end{pmatrix} \, , \]
with $\tmatrix_1'$ positive semidefinite of size $r$, and $\tmatrix_1''=\tmatrix_1^{\on{e}}$ positive definite of size and rank $r_1$. Then it is readily verified that both $\on{rec}( f')$ and $\on{rec} (\varphi')$ are equal to the linear function $x_1\mu_1 = x_1 \tvector_1^t \tmatrix_1 \tvector_1 = \ov{f}(x_1,0,\ldots,0)$. Note that
\[ \varphi_0(-\log|q_1|,\ldots,-\log|q_k|;\underline{q})
= \varphi'(-\log|q_1|;\underline{r},\underline{q}) - f'(-\log|q_1|;\underline{r}) \]
on $V_{\epsilon'} \cap X$. We are done once we show that $ \varphi' - f'$ extends continuously over $\ov{ \bb{R}_{>\kappa} } \times K''$. Following
Theorem \ref{thm:main_technical}.\ref{main_technical_k=1} we have that both
$ \varphi' - \on{rec}( \varphi')$ and $f' - \on{rec} (f')$ extend continuously over $\ov{ \bb{R}_{>\kappa} } \times K''$. As $\on{rec}(\varphi') = \on{rec}( f')$ we find the required extension result. 

The second item of Theorem \ref{singbiext} is clear. As $f_s$ is up to a linear form the recession function of a normlike function we have that $f_s$ is convex, and by Proposition \ref{extendconvex} that $f_s$ extends as a convex, continuous homogeneous weight one function $\ov{f}_s \colon \bb{R}_{\geq 0}^k \to \bb{R}$. This finally proves items (3) and (4) of Theorem \ref{singbiext}.
\end{proof}

\subsection{The Lear extension made explicit}

Let $s$ be an admissible section of $\ca{P} = \ca P_\nu$. We recall that this means that $s$
corresponds to an admissible biextension variation of mixed Hodge structures on $X$. Then $s$
can also be seen as a rational section of the Lear extension
$\Lear{ \ca{P},\aabs{-} }_{\ov{X}}$ of $\ca{P}$ over $\ov{X}$. We can now compute the global
$\bb{Q}$-divisor 
$\on{div}_{\ov{X}}(s)$ that represents 
$\Lear{ \ca{P},\aabs{-} }_{\ov{X}}$. We do
this after a little digression.

Write $U=U_{\epsilon'}$, $V=V_{\epsilon''}$ to reduce notation. 
Let $\mu_1,\ldots,\mu_k \in
\bb{Q}$ be the coefficients of $\varphi$ (see end of Section \ref{propertiesnormlike} for the definition), and $\nu_i = \mathrm{ord}_{D_i} h$
for $i=1,\ldots,k$, and $a_i = \mu_i + \nu_i$. Here $D_i$ is the
divisor on $\ov{X}$ given locally on $U$ by $q_i=0$. We obtain from
the above proof that 
\begin{equation} \label{definepsi} -\log \aabs{s}(\underline{q}) =
  -a_1 \log|q_1| + \psi_1(\underline{q})
\end{equation}
on $V \cap X$ where $\psi_1(\underline{q})$ extends continuously over
$V$. 

We say that $p \in \ov{X}$ is of \emph{depth} $k$ if $p$ is on precisely $k$
of the irreducible divisors $D_i$. The set $\Sigma_k$ of points of
depth $k$ on $\ov{X}$ is a locally closed subset of $\ov{X}$ and for
$k \ge 1$ they yield a stratification of $D = \ov{X} \setminus
X$.
For $p \in \Sigma_k$ take a coordinate neighborhood
$U \subset \ov{X}$ such that $p=(0,\ldots,0)$ and $D \cap X $ is
given by the equation $q_1\cdots q_k=0$. 
Theorem
\ref{singbiext} yields an 
associated homogeneous weight-one function 
$f_{p,s} \in \bb{Q}(x_1,\ldots,x_k)$.

\begin{lemma} \label{locconstant} The map $\Sigma_k \ra \bb{Q}(x_1,\ldots,x_k)$ given by $p
  \mapsto f_{p,s}$ is locally constant. 
\end{lemma}
\begin{proof} Take $p$, $U$ as above and let
  $y=(0,\ldots,0,y_{k+1},\ldots,y_n)\in U$ be another point of depth
  $k$. Let $q_i'=q_i$ for $i=1,\ldots, k$, $q_i'=q_i-y_i$ for
  $i=k+1,\ldots,n$. Then $\underline{q}'$ are coordinates centered
  around $y$ and we have 
\begin{equation*}
\begin{split}
 -\log \aabs{s} & =  f_{p,s}(-\log|q_1|,\ldots,-\log|q_k|) +
 \psi_p(\underline{q}) \\
 & = f_{y,s}(-\log|q_1'|,\ldots,-\log|q_k'|) + \psi_y(\underline{q}') \\
 & = f_{y,s}(-\log|q_1|,\ldots,-\log|q_k|) + \psi_y(q_i-y_i)
\end{split}
\end{equation*}
on $U \cap X$ with $\psi_p$, $\psi_y$ bounded on $U \cap X$. We find
that $f_{p,s}-f_{y,s}$ is bounded on $\bb{R}_{>\kappa}^k$ and,
being homogeneous of weight one, it vanishes identically.  
\end{proof}

In order to compute the divisor $\on{div}_{\ov X}(s)$ that
represents the Lear
extension of $\ca{P}_\nu$ over $\ov{X}$ we are interested in the behavior
of the function
$
f_s \colon \Sigma_1  \to
\bb{Q}(x)
$ obtained from Lemma \ref{locconstant} by restricting to $k=1$.
Note that $\Sigma_1 = D \setminus D^{\mathrm{sing}}$. Let $D=\bigcup
_{\alpha =1}^{d}D_{\alpha }$ be the decomposition of $D$ into irreducible
components. Take any irreducible component $D_{\alpha }$. Since
$D_{\alpha}\setminus D^{\mathrm{sing}}$ is connected, we deduce from Lemma
\ref{locconstant}  that the function 
$$f_{s,\alpha } \colon D_\alpha \setminus D^{\mathrm{sing}} \to \bb{Q}(x)$$ is constant.
Its value is a homogeneous linear function which we write as $f_{s,\alpha}(x)=a_\alpha x$, with $a_\alpha  \in \bb{Q}$. In this notation we find:
\begin{corollary} \label{learext_explicit}
Let $s$ be a section of $\ca{P}$ corresponding to an admissible biextension variation on $X$. Let
$L=\Lear{ \ca{P},\aabs{-} }_{\ov{X}}$ be the Lear extension of
$\ca{P}$ over $\ov{X}$. Then $L$ is represented by the
$\bb{Q}$-divisor
\[ 
\mathrm{div}_{\ov{X}}(s)= \sum_{\alpha=1}^d a_\alpha D_\alpha \]
on $\ov{X}$. 
\end{corollary}

\subsection{Local integrability}

Our next task is to investigate $\partial \bar{\partial} \log \aabs{s}$
over curves.
\begin{proof}[Proof of Theorem \ref{theorem:local_integrability_over_curves}] 
We use the estimates from Theorem \ref{thm:main_technical}.\ref{main_technical_k=1}. We assume $k=n=1$, but otherwise keep the
notation and assumptions from Section \ref{singularity}. In particular
we have the normlike function $\varphi(x_1,q_1)$ on $\bb{R}_{>\kappa}
\times \Delta_\epsilon$ and the associated recession function
$f=\on{rec}( \varphi)$ on $\bb{R}_{>0}$. Put $\varphi_0=\varphi-f$. 
Put $\varphi_1(q_1)=\varphi_0(-\log|q_1|,q_1)$. By Corollary
\ref{explicitnorm} on $U_\epsilon \cap X$, noting that $f$ is linear,
we have 
\[ -\log \aabs{s}(q_1) = -\log|h|(q_1) + \varphi_1(q_1)  \]
for some meromorphic function $h$. Note that
\[ \partial \varphi_1 = -\frac{1}{2}\frac{\partial \varphi_0}{\partial
  x_1} \frac{d q_1}{q_1}
+ \frac{\partial \varphi_0}{\partial q_1} dq_1 . \]
Here $\partial \varphi_0/\partial q_1$ is smooth and bounded on
$\ov{U_{\epsilon'}}$, and by 
Theorem \ref{thm:main_technical}.\ref{main_technical_k=1} we have a constant $c_1$ such that 
\[ \left| \frac{\partial \varphi_0}{\partial x_1} \right| \leq c_1 \cdot
x_1^{-2} \, . \]
Hence for a smooth vector field $T$ with bounded coefficients we find a constant $c_2$ such that
\[ \abs{ \partial \varphi_1 (T) } \leq c_2 \cdot \frac{1}{(-\log|q_1|)^2|q_1|}  \]
on $U_\epsilon \cap X$.
A similar argument yields
\[ \abs{ \bar{\partial} \varphi_1 (T) } \leq c_2 \cdot \frac{1}{(-\log|q_1|)^2|q_1|}  \]
on $U_\epsilon \cap X$. In particular, there is a constant $c_3$ such that
\begin{displaymath}
  \left\| \int_{\partial U_{\epsilon }}\partial \varphi_1\right\|\le c_3
  \frac{\epsilon }{(\log \epsilon )^2 \epsilon }. 
\end{displaymath}
 Thus the residue $\on{res}_{0}(\partial \varphi_{1})$ of $\partial \varphi_{1}$ at zero is zero.

Next, there exists a smooth $(1,1)$-form $\zeta$ on $U_\epsilon$ such that
\[ \partial \bar{\partial} \varphi_1 = \frac{1}{4} \frac{\partial^2 \varphi_0}{\partial x_1^2} \frac {1}{|q_1|^2} dq_1 d \ov{q_1} + \zeta \, . \]
By Theorem \ref{thm:main_technical}.\ref{main_technical_k=1} we have a constant $c_4$ such that
\[ \left| \frac{\partial^2 \varphi_0}{\partial x_1^2} \right| \leq c_4 \cdot x_1^{-3} \, . \]
Hence for smooth vector fields $T, U$ with bounded coefficients we find a constant $c_5$ and an estimate
\[ \abs{ \partial \bar{\partial} \varphi_1 (T,U) } \leq c_5 \cdot \frac{1}{(-\log|q_1|)^3|q_1|^2}  \]
on $U_\epsilon \cap X$. This shows that $\partial \bar{\partial} \varphi_1$ is locally integrable on $U_\epsilon$. 
\end{proof}

\subsection{Effectivity of the height jump divisor}

In this section we prove Theorem \ref{theorem:effectivity}. 
We continue again with the notation as in Section \ref{singularity}. In particular we have $U=U_\epsilon$, we have $s$ an admissible section of $\mathcal{P}_\nu$ on $U \cap X$, and $f_s \colon \bb{R}_{>0}^k \to \bb{R}$ the associated homogeneous weight one function such that
\[ -\log\|s\| - f_s(-\log|q_1|,\ldots,-\log|q_k|) \]
is bounded on $U \cap X$ and extends continuously over $\ov X \setminus D^\mathrm{sing}$. Moreover $f_s$ extends as a convex homogeneous weight one function $\ov{f}_s \colon \bb{R}_{\geq 0}^k \to \bb{R}$ (cf. Theorem \ref{singbiext}). It is clear that a convex homogeneous weight one function is subadditive, hence we have the estimate
\begin{equation} \label{subadditivity}
\ov{f}_s(x_1, \ldots, x_k) \leq \sum_{i=1}^k \ov{f}_s(0,\ldots,0,x_i,0,\ldots,0) 
\end{equation}
on $\bb{R}^k_{\geq 0}$.

Now let $\ov \phi\colon\ov{C}\ra\ov{X}$ be a map from a smooth curve, sending a point $0$ in $\ov{C}$ to $p=(0,\ldots,0)$, and such that there exists an open neighbourhood $V$ of $0$ in $\ov{C}$ such that $\ov \phi$ maps $V$ into $U$. We also assume that $\ov \phi$ does not map $V$ into $D$. Then $\ov \phi$ is given locally at $0 \in \ov C$ by 
\begin{equation*}
\ov \phi(t) = (t^{m_1}u_1, \ldots, t^{m_i}u_i, \ldots) \, , 
\end{equation*}
where $t$ is a local coordinate on $\ov{C}$ at $0$, the $m_i$ are non-negative integers, and $u_i$ are units. Write $\phi$ for the restriction of $\ov \phi$ to $V \setminus \{0\}$.

\begin{proposition} \label{learpullback}
We have an equality of $\qq$-divisors on $V$:
\begin{equation*}
\on{div}\left( \phi^*s\right)|_V = \ov{f}_s(m_1, \ldots, m_k)\cdot[0] \, ,
\end{equation*}
where $\phi^*s$ is viewed as a rational section of the Lear extension $\Lear{\phi^*(\ca{P}_\nu, \aabs{-})}_V$. 
\end{proposition}
\begin{proof}  It suffices to show that
\[ -\log \aabs{\phi^*s} \sim -\ov{f}_s(m_1,\ldots,m_k)\log|t| \]
on $V \setminus \{ 0\}$, where $\sim$ denotes that the difference is bounded and extends continuously over $V$. As by Theorem \ref{singbiext}
\[-\log \aabs{s} - f_s(-\log|q_1|,\ldots,-\log|q_k|)   \] 
is bounded on $U \cap X$ we obtain the boundedness by pullback along $\phi$.
The continuous extendability over $V$ then follows from the boundedness combined with the existence of a Lear extension for $\phi^*(\ca{P}_\nu, \aabs{-})$.   
\end{proof}
\begin{proposition} \label{pullbacklear}
We have an equality of divisors on $V$:
\begin{equation*}
\on \phi^*(\on{div}_{\ov{X}}(s)) =  \sum_{i = 1}^k\ov{f}_s(0,\ldots,0,m_i,0,\ldots,0)\cdot [0] \, ,
\end{equation*}
where $s$ is viewed as a rational section of the Lear extension $\Lear{\ca{P}_\nu, \aabs{-}}_U$.
\end{proposition} 
\begin{proof} This follows immediately from Corollary \ref{learext_explicit}.
\end{proof}
\begin{proof}[Proof of Theorem \ref{theorem:effectivity}]
Combining Propositions \ref{learpullback} and \ref{pullbacklear} one sees that the line bundle
\[ \Lear{\phi^*(\ca{P}_\nu, \aabs{-})}_{\ov{C}}^{\otimes -1} \otimes \ov{\phi}^*\Lear{\ca{P}_\nu, \aabs{-}}_{\ov{X}} \]
has a canonical non-zero rational section, whose divisor is
\[ \left( -\ov{f}_s(m_1, \ldots, m_k) + \sum_{i=1}^k \ov{f}_s(0,\ldots,0,m_i,0,\ldots,0) \right) \cdot[0] \]
on $V$, which is indeed independent of the choice of rational section $s$. This divisor is effective by the subadditivity of $f_s$ expressed by inequality (\ref{subadditivity}). In particular the section is global. 
\end{proof}

\section*{Acknowledgements}

We would like to thank R. Hain and
G. Pearlstein for several discussions and useful hints. We also are
very grateful to the referees 
for a number of important corrections and clarifications. Finally we would
like to thank the hospitality of the Mathematical Institute of Leiden
University and the Instituto de Ciencias Matem\'aticas where the
authors could meet to work on this paper.

The first author has been partially supported by the MINECO research projects
  MTM2013-42135-P and MTM2016-79400-P and ICMAT Severo Ochoa project
  SEV-2015-0554 and 
  the DFG project SFB 1085 ``Higher Invariants''.


\begin{thebibliography}{99}

\bibitem{abbf} O. Amini, S. Bloch, J. Burgos Gil, J. Fres\'an, \emph{Feynman amplitudes and limits of heights}. Izv. Math. 80 (2016), no. 5, 813--848.

\bibitem{Asakura} M. Asakura, \emph{Motives and algebraic de Rham
  cohomology}. In: The arithmetic and geometry of algebraic
    cycles (Banff, AB, 1998), 133--154, CRM Proc. Lecture Notes 24, 2000.  

\bibitem{bhdj} O. Biesel, D. Holmes, R. de Jong, \emph{N\'eron models
    and the height jump divisor}. Trans. AMS 369 (2017), 8685--8723.

\bibitem{bl} C. Birkenhake, H. Lange, \emph{Complex Abelian
    Varieties}. Second edition. Grundlehren der Mathematischen
  Wissenschaften 302. Springer Verlag, Berlin. 
  
\bibitem{boyd} S. Boyd, L. Vandenberghe, \emph{Convex
    Optimization}. Cambridge University Press, 2009. 

\bibitem{brospearl} P. Brosnan, G. Pearlstein, \emph{Jumps in the
    archimedean height}. Preprint, arxiv:1701.05527.

  \bibitem{BGHdJ2} J.I. Burgos Gil, D. Holmes, R. de Jong,
    \emph{Positivity of the height jump divisor}. IMRN, https://doi.org/10.1093/imrn/rnx169

\bibitem{bkk} J. I. Burgos Gil, J. Kramer, U. K\"uhn, \emph{The
    singularities of the invariant metric on the line bundle of Jacobi
    forms}. Recent Advances in 
  Hodge Theory 45--77, (M. Kerr and G. Pearlstein eds.) London Math. Soc. Lecture
  Note Ser. 427, Cambridge Univ. Press. Cambridge, 2016.

\bibitem{cat} E. Cattani, \emph{Mixed Hodge structures, compactifications
and monodromy weight filtrations}. In: Topics in Transcendental Algebraic Geometry 75--100, (P. Griffiths ed.), 
Ann. of Math. Stud. 106, Princeton Univ. Press, Princeton, NJ, 1984. 
  
\bibitem{de} P. Deligne, \emph{Le d\'eterminant de la
    cohomologie}. In: Current trends in arithmetical algebraic
  geometry (Arcata, Calif., 1985), Contemp. Math. 67 (1987), 93--177. 

\bibitem{dediff} P. Deligne, \emph{Equations diff\'erentielles \`a
    points singuliers r\'eguliers}, Lecture Notes in Math., 163, Springer Verlag
  Berlin, 1970.

  \bibitem{SGA7.1} A. Grothendieck. \emph{{G}roupes de monodromie en
      g\'eom\'etrie alg\'ebrique}, Lecture Notes in Math., 288, Springer Verlag
  Berlin, 1972.
  
\bibitem{hainbiext} R. Hain, \emph{Biextensions and heights associated to curves of odd genus}. Duke Math. J. 61 (1990), 859--898.

\bibitem{hain_normal} R. Hain, \emph{Normal functions and the geometry of moduli spaces of curves}. In: G. Farkas and I. Morrison (eds.), Handbook of Moduli, Volume I.  Advanced Lectures in Mathematics, Volume XXIV, International Press, Boston, 2013.

\bibitem{hp} T. Hayama, G. Pearlstein, \emph{Asymptotics of degenerations of mixed Hodge structures}. Adv. Math. 273 (2015), 380--420.

\bibitem{hdj} D. Holmes, R. de Jong, \emph{Asymptotics of the N\'eron height pairing}. Math. Res. Lett. 22 no. 5 (2015), 1337--1371.  

\bibitem{ka} M. Kashiwara, \emph{A study of variation of mixed Hodge structure}. Publ. RIMS, Kyoto Univ. 22 (1986), 991--1024.

\bibitem{knu} K. Kato, C. Nakayama, S. Usui, \emph{SL(2)-orbit theorem for degeneration of mixed Hodge structure}. J. Algebraic Geom. 17 (2008), no. 3, 401--479. 

\bibitem{lear} D. Lear, \emph{Extensions of normal functions and asymptotics of the height pairing}. PhD thesis, University of Washington, 1990.

\bibitem{pearlpeters} G. Pearlstein, C. Peters, \emph{Differential
    geometry of the mixed Hodge metric}. Preprint arxiv:1407.4082. To appear in Comm. Analysis and
  Geom. 

\bibitem{pearldiff} G. Pearlstein, \emph{$\mathrm{SL}_2$-orbits and degenerations of mixed Hodge structure}. J. Diff. Geometry 74 (2006), 1--67.

\bibitem{pearlhiggs} G. Pearlstein, \emph{Variations of mixed
    Hodge structure, Higgs fields, and quantum cohomology}. 
manuscripta math. 102 (2000), no. 3, 269--310.  

\bibitem{Saito1} M. Saito, \emph{Modules de Hodge polarisables}.
  Publ. Res. Inst. Math. Sci. 24 (1988), 849--995.

\bibitem{Saito2} M. Saito, \emph{Mixed Hodge modules}.
  Publ. Res. Inst. Math. Sci. 26 (1990), 221-333.

\bibitem{ps}  C. Peters, J. Steenbrink, \emph{Mixed Hodge structures}. Ergebnisse der Mathematik und ihrer
  Grenzgebiete. 3. Folge. A Series of Modern Surveys in Mathematics,
  52. Springer Verlag, Berlin, 2008. 

\bibitem{Roc70} R.~T. Rockafellar, \emph{Convex analysis}, Princeton Math. Series, vol.~28, Princeton Univ. Press, 1970.

\bibitem{sz} J. Steenbrink, S. Zucker, \emph{Variation of mixed Hodge structure. I}. Invent. Math. 80 (1985), 489--542. 

\end{thebibliography}
\end{document}